\documentclass[11pt]{amsart}
\usepackage{hyperref,comment}
\usepackage{amssymb}
\usepackage{amsmath}
\usepackage{amsthm}
\usepackage{epsfig}
\usepackage{mathrsfs}
\usepackage{xcolor}
\usepackage{graphicx}
\usepackage{color}

\usepackage{amsthm, amsxtra}

\numberwithin{equation}{section} 

\newtheorem{theorem}{Theorem}[section]

\newtheorem{remark}[theorem]{Remark}
\newtheorem{lemma}[theorem]{Lemma}

\newtheorem{proposition}[theorem]{Proposition}
\newtheorem{corollary}[theorem]{Corollary}

\newtheorem*{question*}{Question}

\newtheorem*{problem*}{Problem}
\newtheorem*{conjecture}{Conjecture}

\newtheorem{definition}[theorem]{Definition}

\definecolor{darkgreen}{rgb}{0,0.4,0}

\newcommand{\R}{\mathbb{R}}

\newcommand{\E}{\mathbb{E}}

\usepackage{soul}

\title[Existence for the supercooled Stefan problem]{Existence for the supercooled Stefan problem in general dimensions}
\author{Sunhi Choi, Inwon C. Kim, \& Young-Heon KIm}
\date{}

\address{Department of Mathematics\\ University of Arizona\\ Tucson, AZ, USA}
\email{schoi@math.arizona.edu}

\address{Department of Mathematics\\ University of California at Los Angeles\\ Los Angeles, CA, USA}
\email{ikim@math.ucla.edu}

\address{
Department of Mathematics\\ University of British Columbia\\ Vancouver, V6T 1Z2 Canada
}
\email{yhkim@math.ubc.ca}

\thanks{IK is partially supported by NSF (DMS 2153254).  YHK is partially supported by  the 
Natural Sciences and Engineering Research Council of Canada (NSERC),  with Discovery Grant RGPIN-2019-03926, as well as  Exploration Grant  (NFRFE-2019-00944)  from the New Frontiers in Research Fund (NFRF). 
 YHK is also a member of the Kantorovich Initiative (KI) that is supported by PIMS Research Network (PRN) program. We thank PIMS for their generous support.
Part of this work is completed during YHK's visit at Korea Advanced Institute of Science and Technology (KAIST), and he thanks their hospitality and excellent environment.
\copyright 2024 by the author. All rights reserved. 
}

\begin{document}

\maketitle

\begin{abstract}
  We prove the global-time existence of weak solutions to the supercooled Stefan problem. Our result holds in general space dimensions and with a general class of initial data. In addition, our solution is maximal in the sense of a certain stochastic order, among all comparable weak solutions starting from the same initial data.

Our approach is based on a free target optimization problem for Brownian stopping times,  where the main idea is to introduce a superharmonic cost function in the optimization problem. We will show that our choice of the cost function causes the target measure to accumulate  near the prescribed domain boundary as much as possible. A central ingredient in our proof lies in the usage of dual problem: we prove the dual attainment and use the dual optimal solution  to characterize the primal optimal solution. It follows in turn that the underlying particle dynamics yield a solution to the supercooled Stefan problem.

 \end{abstract}

\tableofcontents
\section{Introduction}\label{sec:intro}

The main objective of this paper is to prove  global existence results for the supercooled Stefan problem $(St)$ in general dimensions, for a class of general initial data.

\medskip
 
 The problem $(St)$  describes freezing of supercooled water into ice, and can be  written as 
 \begin{equation}\label{eq:stefan}
(\eta - \chi_{\{\eta>0\}})_t - \frac{1}{2}\Delta \eta = 0 \hbox{ in } \R^d \times (0,\infty).
\end{equation} 
Here $-\eta$ denotes the temperature of the supercooled water \cite{Ste90, Ste91}. The  term $\chi_{\{\eta>0\}}$ represents the unit amount of change in latent heat when water turns into ice. Heuristically, the problem corresponds to an interacting particle system where active particles generate Brownian motion in the supercooled water region, stop at the boundary of the region and turn  into the ice. The dynamics of the {\it free boundary }$\partial\{\eta>0\}$ is the central interest in this problem.

\medskip

The problem is famously ill-posed with general initial data: most known examples and results are in one space dimension.  The available existence results in multiple space dimensions are limited to either self-similar ones \cite{HP19} or the ones with specially chosen initial data by an elliptic variational inequality \cite{DF84, Kim-Kim21}. 
 This is in contrast to the {\it stable Stefan problem} that usually describes ice melting into water, which has been extensively studied in the literature; see e.g. \cite{friedman68}, \cite{caffarelli77}, \cite{CK10}, \cite{FRS24} and the book \cite{Meirmanov-book}.

\medskip
 
 In one space dimension, 
it is well-known that solutions of \eqref{eq:stefan} develop finite-time discontinuities of their free boundary \cite{She70, HV96, CS96, CK09, CK12},  when the initial data $\eta(x, 0)$ is larger than $1$ near the initial free boundary $\partial\{\eta(\cdot,0)>0\}$.
Intuitively speaking, such discontinuity is due to the excess of active particles present near the water-ice interface, which generates infinite rate of freezing. The jump discontinuity of the free boundaries is the main source of non-existence and non-uniqueness of weak solutions in one space dimension (see for instance the discussion in \cite{CK09, DNS19}). On the other hand, still in one space dimension, for initial data that is small and smooth in its support, it is known that the free boundary does not jump: indeed there exists a unique smooth solution, see for instance \cite{FP81, FPH83, CK12}. 
See also the discussion of this dichotomy with the critical density of size $1$ near the initial free boundary, in terms of one-dimensional particle systems in  \cite{DT19}.  
\medskip

These results motivate our investigation on the global existence of weak solutions for  $(St)$  in general space dimensions. Roughly speaking, our result shows that such existence result holds if there is enough room for the initial data to diffuse within its initial support, to form a unit-density ice at the end of the freezing process.
More precisely, we present a result that addresses all initial data that are compactly supported and less than $1$, and at the same time gives a general condition that allows initial data that are larger than $1$ away from its support boundary. 
\medskip

To our best knowledge, ours is the first global-time existence result for solutions $(St)$ that starts from a given initial data, in general space dimensions: in fact the generality of initial data makes our result new even in the case of one space dimension.  Our approach builds on the variational construction for solutions of $(St)$ based on particle dynamics initiated in \cite{Kim-Kim21}, where a partial existence result has been established. 

\medskip

Before stating the main theorem, we first explain our condition on the initial data $\eta_0$. The condition is based on the {\it subharmonic order} $\leq_{SH}$ between measures, which is a natural stochastic order for the underlying interacting particle system for our problem. Roughly speaking the ordering states that there is a Brownian motion with a stopping time that starts from one distribution and ends with another: see Section~\ref{sec:prelim},  in particular \eqref{eligibility}. 
\begin{definition}\label{def:C0}
A bounded, compactly supported function $\eta_0$ satisfies the condition $(C_0)$ if the following holds:
 \begin{itemize}
 \item[$(C_0)$:]  There exists a constant $0< \delta <1$, such that 
\begin{align*} \eta_0 \leq_{SH} \nu
\hbox{ with respect to $U:=\{\eta_0>0\}$, \quad for some  $\nu$ with  }
 \nu \le \delta \chi_U .
\end{align*}
\end{itemize}

\end{definition}

  The condition $(C_0)$ is natural in terms of the  particle dynamics associated with the freezing process in $(St)$, where the (negative) heat particles follow the Brownian motion in the supercooled water region, hit the ice front, and become frozen. In this context $\nu$ relates to the final distribution of frozen ice that are generated by those Brownian particles.  The inequality $\nu \le \delta \chi_{\{\eta_0 >0\}}$ corresponds to the upper bound for the capacity of the latent heat in ice,  is inspired by the critical threshold $1$ for the constant density observed in the literature. 
 It is worth pointing out that any initial data which is equal or larger than $1$ near the boundary of its positive set cannot satisfy  $(C_0)$, since then there is no room for diffusing such a mass  by the Brownian motion to another measure with less than unit density supported in the same positive set. On the other hand, it is important to point out that
\begin{align*}
 \hbox{if $\eta_0 \le \delta$ for some constant $\delta <1$, then it automatically satisfies $(C_0)$.}
\end{align*}
As we will see later, the condition $(C_0)$ presents a plethora of initial data that are larger than $1$ in the interior of its positive set.

\medskip

Now we are ready to state our main existence theorem.

\begin{theorem}\label{thm:A}
Let $\eta_0$ satisfy $(C_0)$, with its positive set $\{\eta_0>0\}$ open, bounded, and with locally Lipschitz boundary. In addition suppose that $|\{\eta_0=1\}|=0$.

Then there exists a weak solution $\eta \in L^1(\R^d\times [0,\infty))$ to  \eqref{eq:stefan} with initial data
$\eta_0$ that satisfies the following additional properties:
\begin{itemize}
\item $\{\eta>0\}$ is open and its {\it initial domain} $\limsup_{t\to 0^+} \{\eta(\cdot,t)>0\}$ equals $\{\eta_0>0\}$, possibly minus a measure zero set (see Section~\ref{observation00}).
\item  $\eta$ is maximal in the sense of Definition~\ref{def:maximal} below.
\end{itemize}
\end{theorem}

 The proof of Theorem~\ref{thm:A}  will be presented in Section~\ref{sec:main}.  We remark that our assumption on the Lipschitz geometry of $U=\{\eta_0>0\}$ can be generalized further, for instance, to any bounded simply  connected open set $U$ in two space dimension, and any open set satisfying exterior cone/ball condition in higher  dimensions: see Remark~\ref{assumption:U}.  
 
 \medskip
 
 Our notion of weak solution, given in  Definition~\ref{def:weak-solution}, is the standard version given by integration by parts. Our notion of maximality bears similarity to the `minimal jump' solution discussed in \cite{DNS19} for one space dimension. For a weak solution $\eta$ to \eqref{eq:stefan}, we define the {\it transition zone} $\Sigma(\eta)$ to be the set of points $x$ where the supercooled water turns into ice,
that is, 
$$\Sigma(\eta): = \{x \ | \ \eta(x,t)=0 \hbox{ for some } 0< t < \infty\} \cap \{\eta_0>0\}.$$ 

 In the definition below we consider the maximality of $\eta$ in terms of its transition zone, with respect to the subharmonic order $\leq_{SH}$.
 
\begin{definition}[Maximal solution]\label{def:maximal}
 A weak solution $\eta$ to \eqref{eq:stefan} is  {\em maximal} if 
 for any other weak solution $\tilde \eta$ with the same initial data $\eta_0$ and initial domain $\{\eta_0>0\}$,
 \begin{align*}
\chi_{\Sigma(\eta)} \leq_{SH}  \chi_{\Sigma(\tilde{\eta})}   \hbox{ implies } \eta =\tilde \eta.
\end{align*}
\end{definition}

\medskip

Let us now briefly discuss optimality of our result. The additional conditions $|\{\eta_0=1\}|=0$ in the above theorem, and the constant $\delta<1$ are necessary for our approach (see Remark~\ref{rmk:condition} and Remark~\ref{rmk:delta}), but  they seem  not essential to the existence result. Let us indeed propose what we expect to be a sharp condition for the global existence result:
 \begin{align*}
 (C): \quad \quad & \eta_0 \hbox{  is strictly  subharmonically ordered with  $\nu$, with respect to }\{ \eta_0>0\},\\
& \hbox{ for some  $\nu$ with  }  
 \nu \le  \chi_{\{\eta_0>0\}}.
\end{align*}
By strictly suharmonic order we mean that $\eta_0\leq_{SH} \mu$ with strictly positive stopping time (see Definition~\ref{strict}).  It is indeed a necessary condition as a consequence of the analysis in \cite{Kim-Kim21}, see Theorem~\ref{thm:necessary}. Our conjecture is that the condition is also sufficient:

\begin{conjecture}
There exists a global-time weak solution of $(St)$ with initial data $\eta_0$ and initial domain  $\{\eta_0>0\}$ if and only if  $\eta_0$ satisfies  condition $(C)$.
\end{conjecture}

\medskip

While Theorem~\ref{thm:A}  doesn't fully answer our conjecture, it almost does. Moreover there is a broad class of initial data that meets the condition $(C_0)$. Remark~\ref{example-SO} discusses a simple example of $\eta_0$ satisfying $(C_0)$. Furthermore in Section~\ref{sec:obstacle-to-Stefan} we discuss concrete conditions for $\eta_0$ that fails the sharp condition $(C)$, and conditions for $\eta_0$ that implies  $(C_0)$. These conditions generate many nontrivial examples for the initial data our result holds, beyond those that are smaller than $1$: see Corollarys~\ref{cor:nonexistence} and \ref{cor:existence} and remarks therein.

\medskip

We expect that the evolution of the free boundary in Theorem~\ref{thm:A} will generate many intriguing irregularities that go much beyond topological changes of the positive set $\{\eta>0\}$.
Even with smooth and small initial data $\eta_0$, the evolution of the positive set $\{ \eta>0\}$ appears to be irregular for most choices of the initial domain $U:=\{\eta_0>0\}$, for space dimension larger than $1$. For instance, one may expect that the set $\{ \eta(x,t)>0\}$ continuously decreases from $U$ initially  to a compact subset of $U$ eventually. Interestingly, however, one can check that this is not the case if  $U$ has any non-smooth ($C^\infty $) boundary point: see Section~\ref{observation00}. 

\medskip

Let us briefly describe our approach in the paper. We consider a novel optimization problem \eqref{eqn:primal-problem}, where we find a free target measure supported inside a given domain $U$,  that minimizes the expected value of a given superharmonic function $u$. The admissible class of measures are those generated by the Brownian stopping times from a given initial distribution $\eta_0$,  bounded by the characteristic function $\chi_U$.  The main step in our approach is then to show that such an optimal measure $\nu^*$ is a characteristic function. Once we establish this, it follows from \cite{Kim-Kim21} that  the optimal measure $\nu^*$ corresponds to the final distribution of ice in $(St)$ after the freezing procedure is  complete. For our analysis it is important to observe that, superharmonicity of the cost $u$ makes the Brownian particles to move as much as possible, as the longer the stopping time, the smaller the corresponding expected value of $u$; this principle is manifested in  that the final optimal distribution $\nu^*$ is to be as close to the boundary of the initial domain as possible. See Section~\ref{sec:primal}, where we present the connection between our optimization problem and the Stefan problem $(St)$, which leads to the discussion of condition  $(C)$.

\medskip

The key step   in our method  
 is  to consider the dual problem \eqref{eqn:dual-problem-new}, for which we prove no-duality-gap (Theorem~\ref{th:dual}) and dual attainment (Theorem~\ref{th:dual-attain}). The dual formulation is crucial in showing the saturation property, Theorem~\ref{th:saturation}, of the optimal measure, where the saturated region is determined by the dual optimal solution in an explicit way.
  We then use the results of \cite{Kim-Kim21} to connect the optimal measure with a solution to the Stefan problem to prove our main theorems in Section~\ref{sec:main}.

\medskip

 There are important questions to be answered in future work. For example we suspect that  our class of weak solutions given in Theorem~\ref{th:Stefan-weak} feature continuous evolution of the free boundary (see Remark~\ref{continuity}),  though this remains to be confirmed.  Moreover, we  also conjecture that uniqueness holds for maximal solutions given in Definition~\ref{def:maximal},  for a given pair of initial data and initial domain.
  With our approach, such uniqueness would be equivalent to 
 the statement that  the optimization problem \eqref{eqn:primal-problem} yields the same optimal solution independent of the choice of the cost
  function $u$;  in the present paper we only have uniqueness of the solution for a given superharmonic cost $u$ (see Corollary~\ref{cor:uniqueness-primal}).

 \medskip

\noindent{\bf Acknowledgements.} 
We thank Mathav Murugan and Zhen-Qing Chen for helpful discussions on probabilistic aspects. 
\section{Preliminaries}\label{sec:prelim} 

\subsection{Assumptions}\label{sec:Assumptions and notation}
Throughout the paper we use the following assumptions. 
\begin{itemize}
\item $U$ is a bounded open set in $\R^d$, with locally Lipschitz boundary.
\item For each $r>0$, we denote  $U_r  := \{ x \in U \ | \ d(x, \partial U) > r\}$. 

\item $u: \bar{U} \to \R_{\ge 0}$ is a bounded, smooth function on $\bar{U}$ such that  $$\Delta u < 0 \hbox{ in } U.$$ 
 \item  $f:  \R^d \to \R_{>0}$ is a bounded, positive measurable function. 
 \item $\mu$ has a bounded density supported in $\bar{U}$ and
$$\mu[U] \le \int_{U} f dx.$$

 \item The expression $\tau \oplus \sigma$ is the concatenation of the two stopping times, namely, the Brownian particles first move until $\tau$, and then they further move for the additional time $\sigma$. 

\end{itemize}

The Lipschitz assumption on $\partial U$ is used to ensure that the stopping times associated with $U$ are well approximated with those associated with $U_r$: see Remark~\ref{assumption:U}. While we expect that it can be generalized to a broader class of domains, some very irregular domains may be problematic with our choice of approximation, e.g.,  those with boundaries with positive Lebesgue measure.

\subsection{Randomized stopping times}
 We now define a few standard notions for Brownian motion; here we use the same notation as in
   \cite{Kim-Kim21}.
We make the same assumptions as in \cite{Kim-Kim21} on the probability space. Let $(\Omega, \mathcal{F}, (\mathcal{F}_t)_{t\geq 0}, \mathbb{P})$ be a filtered probability space that incorporates the Brownian motion $W_t: \Omega \to \R^d$.

\medskip

 A random variable $\tau:\Omega\to \R^+$ is called a {\it stopping time} if $\tau^{-1}([0,t])\in \mathcal{F}_t$ for any $t\geq 0$. Stopping times can be regarded as a rule that prescribes when each particle of the process stops moving. An important example of stopping time, to be used frequently in this paper, is the {\it exit time} associated with a domain. Let $W_t$ denote a Brownian motion, and 
 for a given domain $D$ in $\R^d$, let $\tau^D$ denote the exit time for $D$, that is, 
\begin{align*} 
 \tau^D: = \inf \{ t  >0  \ | \ W_t \in D^c\}
\end{align*}
(Here, we use the convention that $\tau^D=0$ for those Brownian paths starting already from $D^c$.)

\medskip

 For our analysis we introduce a relaxed notion of the stopping time, called a {\it randomized stopping time}, which assigns a probability distribution to each particle instructing when to stop. 
 
 \begin{definition} 
Let $\mathcal{M}(\R^+)$ be the collection of probability measures on $\R^+$. A {\it randomized stopping time } $\tau$ is a map from  $\Omega$ to  $\mathcal{M}(\R^+)$ such that 
$$
\hbox{ For each } t\geq 0, \hbox{the map } w\in \Omega \to \tau(w)([0,t]) \hbox{ is } \mathcal{F}_t - \hbox{measurable. }
$$

\end{definition} 

For a randomized stopping time $\tau$ and $\varphi:\R^+\to \R$, we denote
$$
\mathbb{E}[\varphi(\tau)] =\mathbb{E}\left[\int_{\mathbb{R}^+} \varphi(t)\tau(dt)\right]. 
$$

We say that $\nu$ is the distribution of the stopped Brownian motion $W_t$ at the (randomized) stopping time $\tau$, that is $W_{\tau} \sim \nu$, if 
$$
\mathbb{E}[g(W_{\tau})] = \int g(z) d\nu(z)  \quad\hbox{ for any nonnegative $g$ that is measurable in $\R^d$.} 
$$

For more detailed  discussions on randomized stopping time, we refer to \cite{Kim-Kim21}.

\subsection{Superharmonic functions.}

In this section we will collect different notions of superharmonic functions and their relations. Let us first begin with the standard notion.

\begin{definition}\label{def:conventional-superharmonic}

A lower semi-continuous function $\phi: \mathcal{O} \to \R \cup\{-\infty\}$ is {\it superharmonic }if for every $x\in \mathcal{O}$ and for every ball $B$ in $\mathcal{O}$  centered at $x$, it satisfies
$$
\phi(x) \leq \frac{1}{|B|} \int_B \phi (y) dy.
$$
\end{definition}

\medskip

Now, we consider a maximum principle of superharmonic functions via stopping times, which will be crucial in the dual formulation. 
The following is an application of a well-known lemma (see. e.g. \cite[Lemma 2.3 ]{GKP21}):

\begin{lemma}\label{lem:maximum-principle-stopping}
 Let  $\phi$ be a  superharmonic function in $U$. For each $x \in \R^d$, let $W^x_t$ denote a Brownian motion at time $t$ started from $x$ at $t=0$.
  Then, for each $r >0$,  for a.e. $x \in U$,  and for each (randomized) stopping time $\tau$ that satisfies  $\tau \leq \tau^{U_r}$,
\begin{align} \label{eqn:superharmonic}
   \phi (x) \ge \E \left[ \phi(W^x_\tau )\right].
\end{align}
\end{lemma}
\begin{proof}
 Observe that (\ref{eqn:superharmonic}) is 
 trivial for  $x \in U \setminus U_r$ as $\tau^D =0$ by our convention for those Brownian paths started from such $x$. 
For $x \in U_r$, we can apply \cite[Lemma 2.3]{GKP21} as $\phi$ is in $LSC(\overline{U_r})$. 
\end{proof}

Lemma~\ref{lem:maximum-principle-stopping} induces a  simple but important result for us,  Lemma~\ref{lem:stopping-lemmas}, which will be used in  Sections~\ref{sec:duality} and \ref{sec:saturation}. Note that we crucially use the condition $\nu \le f$ in its proof. 
\begin{lemma}\label{lem:stopping-lemmas}
Let $\tau\le \tau^U$ be a (randomized) stopping time with $W_0\sim \mu$ and
  $W_\tau \sim \nu$ for some $\nu$ with $\nu \le f$.
 For each $r>0$, define $\tau_r := \tau\wedge \tau^{U_r}$ ($:= \min \{ \tau, \tau^{U_r} \} $) and let $\nu_r$ be the distribution $W_{\tau_r} \sim \nu_r$ with  $W_0 \sim \mu$. 
 Then, 
 \begin{equation}\label{observation_2-1}
\lim_{r\to 0} \nu_r(E) = \nu(E) \hbox{ for any measurable subset } E\hbox{ of }  U.
\end{equation}
Moreover, if $\phi$ is superharmonic then
\begin{align}\label{eqn:maximum-good}
 \int \phi (x) d\mu(x) \ge \int \phi (y) d\nu(y).
\end{align}

\end{lemma}
\begin{proof}

Observe that $\nu_r \leq \nu$ as measures on $U_r$, because, by the definition  of $\tau_r : = \tau\wedge \tau^{U_r}$, it follows that the only possibility for $\nu_r$ is bigger than $\nu$ (as a measure) is along $\partial U_r$, where $\nu_r$ has singular part. Since $\nu \leq f$ and $f$ is supported on $U$, $\lim_{r \to 0} \nu (U_r^c) =0$. 
  We show that the singular part of $\nu_r$ converges to zero.  
  
  \medskip
  
  For sufficiently small $r>0$, the probability of Brownian motion started from a point $x \in \partial U_r$ to hit $\partial U_{\sqrt{r}}$ before it hits $\partial U$ is bounded above by $C(r/\sqrt{r})^\beta$, for a dimensional constant $\beta>0$ depending on the Lipschitz constant of $U$. This implies 
	$$ \lim_{r \to 0} \nu_r(\partial U_r) \leq \lim_{r \to 0} (\nu(U_{\sqrt{r}}^c) + C r^{\beta/2} ) =0.$$ We thus conclude that $\nu_r$ converges to $\nu$ as a measure on $U$, namely \eqref{observation_2-1}.
In fact the density of  the absolutely continuous part of $\nu_r$ monotonically converges to the density of $\nu$. 

\medskip

Now note that \eqref{eqn:maximum-good} follows by applying \eqref{observation_2-1} to Lemma~\ref{lem:maximum-principle-stopping} because 
\begin{align*}
\int   \phi (x) d\mu(x) \ge \int \E \left[ \phi(W^x_{\tau_r} )\right] d\mu(x) = \int \phi(y ) d\nu_r(y) \to \int \phi(y) d\nu(y) \quad \hbox{ as } r\to 0^+.
\end{align*}
(For the last limit one can use the monotone convergence Theorem~\ref{thm:A}pplied to the absolutely continuous part of $\nu_r$.)
This completes the proof. 
\end{proof}

\medskip

\begin{remark}\label{assumption:U}
 Lemma~\ref{lem:stopping-lemmas} can be proved with more general domains $U$ beyond those with locally Lipschitz boundaries, though we choose not to pursue this for simplicity of presentation.  In particular  our main results in  Sections~\ref{sec:duality} and \ref{sec:saturation}  hold with domains $U$ for which the stopping times $\tau^*_r:=\tau^* \wedge \tau^{U_r}$ converge to $\tau^*$ in terms of their associated distributions when the target distribution has a bounded density: this property is  satisfied for instance,  by $U$ whose boundary is Lipschitz away from a finite number of inward or outward cusps. 
 
 \medskip
 
Lemma~\ref{lem:stopping-lemmas} also holds when $U$ satisfies the exterior cone condition: for any $x \in \partial U$, there is a cone $V \subset U^c$ with vertex at $x$, constant opening and height, and axis $e(x)$ depending on $x$. More generally, the lemma holds under a weaker condition, the exterior ball condition: for any $x \in \partial U$ and small  $r>0$, there is a ball $ B(y, s) \subset U^c \cap B(x,r)$ with $s \approx r$. This condition allows Koch snowflakes as an initial domain $U$. 

\medskip

 Furthermore in two dimension, Lemma~\ref{lem:stopping-lemmas} holds for any simply connected bounded open set $U$: when $d=2$, Beurling's projection theorem \cite[Theorem 9.3]{G-book} implies that the probability of the Brownian motion started from a point $x \in \partial U_r$ to hit $\partial U_{\sqrt{r}}$ before it hits $\partial U$ is bounded above by $C(r)$ with $\lim_{r \to 0} C(r) =0$, and hence the proof of the lemma works. 

\end{remark}

 We now turn to another important definition  relevant to the notion of superharmonicity with stopping times.
Given a measurable function 
$\phi: \bar{U} \to \R$, we define its {\it superharmonic envelope} $\phi^{sp}: \bar{U} \to \R$ as 
\begin{align}\label{eq:sp}
 \phi^{sp} (x) := \sup_{\tau \le \tau^U} \E \left[ \phi(W^x_\tau )\right].
\end{align}
where $\tau$ is any (randomized) stopping time such that $\tau \le \tau^U$  (see e.g. \cite{GKP21}). 

\begin{lemma} \label{lem:LSC}
 For $\phi \in C(\bar U)$,  the function  $\phi^{sp}:  U \to \R$ is lower semi-continuous and satisfies (\ref{eqn:superharmonic}).
\end{lemma}

\begin{proof}

We first show that $\phi^{sp}$ is a measurable function on $U$. 
To see this for each given $\delta>0$, use the measurable selection theorem to get a measurable function $x \in U \to \tau^{\delta, x}$, from $U$ to the space of randomized stopping times endowed with the $\hbox{weak}^*$-topology over the space of paths, such that  
the function $\psi_\delta (x) := \E \left[ \phi(W^x_{\tau^{\delta, x}} )\right] $ satisfies 
\begin{align*}
 \hbox{$ \psi_\delta (x) \ge  \phi^{sp}(x) -\delta$ for each $x \in U$. }
\end{align*}
Since $\psi_\delta \le \phi^{sp}$ from the definition of $\phi^{sp}$, we see that as $\delta \to 0$, $\psi_\delta \to \phi^{sp}$ for each $x$. Therefore, the measurability of $\psi_\delta$ will prove the measurability of $\phi^{sp}$, as a pointwise limit of measurable functions. 
Now since $\phi \in C(\bar U)$, the function $\tau \mapsto \E \left[ \phi(W^x_{\tau})\right]$ is continuous function, therefore the composition $x \mapsto \tau^{\delta, x} \mapsto \E \left[ \phi(W^x_{\tau^{\delta, x}})\right]$ is measurable. This proves the measurability of $\psi_\delta$.

\medskip

Next, the superharmonic mean value property of  $\phi^{sp}$ follows from the dynamic programming principle. For each stopping time $\tau \le \tau^U$,  note that 
$ \phi^{sp}(W^x_\tau ) =  \sup_{\tau \oplus \sigma\le \tau^U} \E \left[  \phi(W^x_{\tau \oplus \sigma} ) \right] $ by the definition \eqref{eq:sp} of the superharmonic envelope. Then 
\begin{align*}
  \E \left[ \phi^{sp}(W^x_\tau )\right] 
  =   \E \left[  \sup_{\tau \oplus \sigma\le \tau^U} \E \left[  \phi(W^x_{\tau \oplus \sigma} ) \right] \right] \le  \E \left[ \phi^{sp}(x)\right] = \phi^{sp} (x). \end{align*}
It follows that $\phi^{sp}$ satisfies \eqref{eqn:superharmonic}  and is lower semi-continuous (see   \cite[Proposition 2.8]{silvestre}).
 \end{proof}

\subsection{Strictly subharmonic order}\label{section:subharmonic_order} 

Two measures $\mu, \nu$ on $\R^d$ with equal volume are said to be in {\it subharmonic order} $\mu \leq_{SH} \nu$, if for each open set $\mathcal{O}$ containing the support of $\mu+\nu$,
$$
\int \varphi(x) d\mu(x) \leq \int \varphi(x)d\nu(x) \quad \hbox{ for each smooth subharmonic function in } \mathcal{O}. 
$$
It is known that for compactly supported $\mu$ and $\nu$, $\mu\leq_{SH} \nu$ if and only if there is a randomized stopping time $\tau$ for the Brownian motion such that 
\begin{equation}\label{eligibility}
W_0\sim \mu \hbox{ and }W_{\tau}\sim \nu \hbox{ with }\mathbb{E}(\tau) <\infty.
\end{equation} 
We say that such an order $\mu\leq_{SH} \nu$ is with respect to $U$ if   such  a (randomized stopping)  time $\tau$ satisfies $\tau\le \tau_U$. 
 For many types of cost functions such as $\mathcal{C}(\tau) = \mathbb{E}(\tau^2)$, there exists an optimal (randomized) stopping time that satisfies \eqref{eligibility} and minimizes the cost (see \cite{beiglboeck2017optimal, PDE19, GKP21}). In fact, the optimal stopping time $\tau^*$ is  a non-randomized stopping time. In our problem we will first optimize the target measure in the primal problem \eqref{eqn:primal-problem} for the given measure $\mu$, and then use the optimal stopping time for the optimal target measure $\nu^*$ to establish the connection to the Stefan problem. 

\medskip

For our analysis we also need a stronger ordering that prevents any instant stopping of the Brownian particles when they follow the optimal stopping time. This is  to ensure that the corresponding solution of the Stefan problem evolves with the initial distribution $\mu$. 

\begin{definition}\label{strict}
 A pair ($\mu$, $\nu$) is {\it strictly subharmonically ordered} if $\mu\leq_{SH} \nu$ and  its optimal stopping time $\tau^*$ with the cost   $\mathcal{C}(\tau) = \mathbb{E}(\tau^2)$, is strictly positive. We also say such an order is with respect to $U$, if this $\tau$ satisfies $\tau \le \tau_U$.
\end{definition}

\begin{remark}\label{example-SO}
Note that if $\mu <1$ and $\mu\le_{SH} \chi_E$ for some $E$, then such a subharmonic order is strict,  due to \cite[Theorem 4.7 (D1)]{PDE19}.
 It is also easy to construct $\mu$ whose density is larger than 1 away from the boundary of its support, which are still strictly subharmonically ordered to $\nu$ for some $\nu <1$. For  a simple example, given $0<\delta<1$ and $M>1$, let $\mu = \delta \chi_{B(0,1)} + M \chi_{B(0,r) }$ for some $r>0$. If $r$ is sufficiently small (i.e., $r \leq r_0(n, \delta, M)$), then $\mu\leq_{SH} \nu:= \delta' \chi_{B(0,1)}$ for some $\delta< \delta'<1$ and its optimal stopping time $\tau^*>0$. 
 For more general examples, see Corollary~\ref{cor:existence} and Remark~\ref{rmk:many-examples}.
\end{remark}
\bigskip

\section{Primal problem}\label{sec:primal}
For $\mu$, $f$ and $u$ as given in Section~\ref{sec:prelim},  let us consider the following optimization problem about stopping times $\tau$ and their target distributions $W_\tau\sim \nu$:
\begin{align}\label{eqn:primal-problem}
 P(\mu, f, u) = \inf \left \{ \int u (x) d\nu(x)  \ \Big| \  \tau \le \tau^U, W_\tau \sim \nu \le f \chi_U \ \& \ W_0 \sim \mu  \right\}.
\end{align}
The constraint $\nu \leq f\chi_U$ means that $\nu$ has a density which is less than or equal to $f \chi_U$ a.e. It is easy to see that $ P(\mu, f, u)$ has an optimal solution.

\begin{theorem}[Existence of a primal solution]\label{th:primal-existene}
The problem $ P(\mu, f, u)$ has an optimal solution if the eligible set of $\nu$ is non-empty. 
\end{theorem}
\begin{proof}
 We consider the stopping times $\tau$, and its corresponding distributions $\nu$,  that are admissible for the problem $P(\mu, f, u)$. Let  $(\tau_i, \nu_i)$ be a  minimizing sequence. 
 We can view $\tau_i$ as randomized stopping times (see e.g. \cite{beiglboeck2017optimal} or \cite{Kim-Kim21}), which  converge via the weak$^*$-topology to a randomized stopping time $\tau_\infty$, with $\tau_\infty$   and its distribution giving  the pair $(\tau_\infty, \nu_\infty)$. Notice that $\nu_i \to \nu_\infty$ in the weak$^*$-topology, therefore we can keep $\nu_\infty \le f\chi_U$ and also 
\begin{align*}
 P(\mu, f, u) =\lim_{i\to\infty} \int u d\nu_i =  \int u d\nu_\infty.
\end{align*}
 From a result of \cite{beiglboeck2017optimal, PDE19}, there is a corresponding (non-randomized) stopping time $\tau^*$ satisfying the conditions of $P(\mu, f, u)$, therefore $\nu_\infty$ is admissible. This completes the proof. 
\end{proof}
\begin{remark}
In Corollary~\ref{cor:uniqueness-primal}, we prove the uniqueness of an optimal solution to the primal problem  $ P(\mu, f, u)$, assuming 
\begin{align}\label{eqn:SHmu}
 &\hbox{  $\mu$ is subharmonically ordered with respect to $U$,} \\\nonumber
  & \qquad \qquad  \hbox{with  some  $\nu$ with $\nu\le\delta f\chi_U$ for some $\delta \in (0,1)$,}
\end{align}
which is a slightly weaker condition than $(C_0)$ in Definition~\ref{def:C0}, as $\mu$ may not be positive everywhere in $U$.  Note that \eqref{eqn:SHmu} is automatically satisfied if $\mu \le \delta f $.

\end{remark}

One of the motivations for \eqref{eqn:primal-problem} is that its optimal target measure will be supported near the boundary of $U$ as much as possible,  due to the super-averaging  property  (\ref{eqn:superharmonic}) of the superharmonic function $u$: the strict superharmonicity $\Delta u <0$ makes the stopping time $\tau$ to be as large as possible. To see this, observe that for $W_0\sim \mu$, $W_{\tau_i} \sim \nu_i$, $i=1,2$, if $\tau_1 \le \tau_2$ then
\begin{align*}
 \int u (x) d\nu_1 \ge \int u(x) d\nu_2, \hbox{ and ``$>$'' if $Prob[\tau_1 < \tau_2] >0$}
\end{align*}
as  we see from Ito's formula
\begin{align*}
 \int u (x) d\nu_i =\int u(x) d\mu(x) + \mathbb{E} \left[ \int_0^{\tau_i} \frac{1}{2}\Delta u (W_t) dt\right].
\end{align*}

For the same reason, one can expect that the optimal strategy for the Brownian particles is to move as much as possible so that the target measure is saturated in its support: see Proposition~\ref{prop:easy-sat} where this argument is presented in the interior of the support of the target measure. A rigorous proof of saturation will be given in Section~\ref{sec:saturation} (Theorem~\ref{th:saturation}), where we present a delicate argument using the dual variable to avoid the difficulty arising from the unknown topological nature of the support for the optimal target measure. 

\medskip

 If the optimal stopping time were strictly positive, together with the saturation property,  we can relate $P(\mu, f, u)$ to the supercooled Stefan problem \eqref{eq:stefan} due to the results achieved in \cite{Kim-Kim21}: see Section~\ref{sec:main}. Due to the fact that the target measure is supported up to the boundary of $U$, it will follow that the optimal stopping time for $P(\mu, f,u)$ provides a global solution to $(St)$ with the initial domain $U$, yielding our Theorem~\ref{thm:A}.  The strict positivity of the stopping time can be ensured for instance if $\mu$ satisfies $|\{\mu=f\}|=0$, due to Theorem 8.3 (a) in \cite{Kim-Kim21}.  Later in the paper we will impose the assumption  \eqref{eqn:SHmu} 
 which  is  also used to obtain a technical compactness estimate for the dual attainment (Theorem~\ref{th:dual-attain}).

 \medskip

Of course to consider $P(\mu, f, u)$, we also need to ensure that the eligible set of $\nu$ is non-empty. This is certainly the case with the condition  \eqref{eqn:SHmu} since then $\mu$ is in the eligible set. Our conjecture says that in principle our approach should hold  under the assumption $(C)$ which says (a) the admissible set of $\nu$ is non-empty and (b) the optimal stopping time is strictly positive.

\section{Duality}\label{sec:duality}

We consider the dual problem:
 \begin{align}\label{eqn:dual-problem-new}
 D(\mu, f, u) =  \sup_{ ( \phi,   \, \psi ) \in \mathcal{D} } \left[\int \phi (y) df (y)  - \int \psi (x)d\mu(x) \right]
\end{align}

where 
\begin{equation}\label{eqn:D-new}
  \mathcal{D} : = \{ (\phi, \psi) \in L^1(U) \ | \ \phi \le 0,\,\,   \psi \ge \phi -u \hbox{ in } U \hbox{ and } \psi \hbox{ is superharmonic 
  }\}.
\end{equation}

We first establish the weak duality:
\begin{theorem}[Weak duality]\label{th:weak-dual}
For $\mu$, $f$, $u$ as given in Section \ref{sec:Assumptions and notation},
\begin{align*}
 D(\mu, f, u) \le P(\mu, f, u).
\end{align*}
\end{theorem}
\begin{proof}
Consider an admissible primal/dual pair $(\nu, (\phi, \psi))$ for $P(\mu, f, u)$ and $D(\mu, f, u)$. Namely, for a stopping time $\tau$  such that 
$ \tau \le \tau^U, W_\tau \sim \nu \le f \chi_U \ \& \ W_0 \sim \mu $ and 
$(\phi, \psi) \in \mathcal{D}$ as in \eqref{eqn:D-new}.
It suffices to show that
\begin{align*}
  \int \phi df - \int \psi d\mu \le  \int ud\nu.
\end{align*}

\medskip

The subharmonic ordering \eqref{eqn:maximum-good} and  the fact $\psi \ge \phi -u$ yield  
\begin{align}\label{observation_1}
 \int \psi d\mu   \ge \int \psi d\nu \ge  \int  (\phi-u)  d\nu.
\end{align}

Now let us consider the following inequalities: 
\begin{align*}
 \int \phi df - \int \psi d\mu
 & \le \int \phi df- \int (\phi -u)d\nu \quad \hbox{ (due to \eqref{observation_1})} \\
 & = \int \phi(df-d\nu) +\int ud\nu \\
 & \le \int u d\nu \quad \hbox{( since $\phi \le 0$, $\nu\le f$ on $U$)}. 
\end{align*}

 Thus  proves the weak duality.
\end{proof}

We in fact have a no-duality-gap result which is  crucial  for the present paper. For its proof, we introduce a modified dual problem:
\begin{align}\label{eqn:dual-problem}
	D'(\mu, f, u) =  \sup_{ ( \phi,   \, \psi ) \in \mathcal{D}' } \left[\int \phi (y) df (y)  - \int \psi (x)d\mu(x) \right]
\end{align}
where 
	\begin{align*}
		\mathcal{D}' : = \{ (\phi, \psi) \in  L^1(U)  \ | \ \phi \le 0 \, \ \&  \ \psi (x) \ge \E\left[(\phi  -u) (W^x_\tau) \right] \hbox{ for any $\tau \le \tau^U$}\}.
\end{align*}
Then since $\mathcal{D}' \subset \mathcal{D}$, we observe 
\begin{align} \label{D'D}
	D'(\mu, f, u) \leq 	D(\mu, f, u).
\end{align}

\begin{theorem}[No Duality Gap]\label{th:dual}
For $\mu$, $f$, $u$  as given in Section \ref{sec:Assumptions and notation}, 
\begin{align}\label{D=D'}
 P(\mu, f, u)= D'(\mu, f, u)
\end{align}
and hence $ P(\mu, f, u)= D'(\mu, f, u) =  D(\mu, f, u)$.
\end{theorem}
Note that the last equality of Theorem~\ref{th:dual} follows from  Theorem~\ref{th:weak-dual} with (\ref{D'D}), (\ref{D=D'}) since 
$$ P(\mu, f, u)= D'(\mu, f, u) \leq  D(\mu, f, u) \leq P(\mu, f, u).$$

Before proving this theorem in Section~\ref{sec:dual-proof}, 
 let us give heuristics that utilize the `min-max' principle, where ``$=$'' is a nonrigorous part:
\begin{align*}
 P(\mu, f, u) &=\inf_{\tau \le \tau^U \& \, W_0 \sim \pi_1 \& W_\tau \sim \pi_2}  \left \{ \int u d\pi_2 + \sup_{ \phi \le 0, \, \psi} \left[\int \phi (df -d\pi_2) + \int \psi ( d\pi_1 - d\mu) \right]   \right\}\\
 & \hbox{``}= \hbox{''}\sup_{ \phi \le 0, \, \psi}  \left\{ \inf_{\tau \le \tau^U \& \, W_0 \sim \pi_1 \& W_\tau \sim \pi_2}   \left[\int u d\pi_2 + \int \phi (df -d\pi_2) + \int \psi (d\pi_1 -d \mu) \right]   \right\}.\\
 \end{align*}
 Then after rearranging,  the last expression becomes
 \begin{align*}
 & = \sup_{ \phi \le 0, \, \psi}  \left\{ \inf_{\tau \le \tau^U \& \, W_0 \sim \pi_1 \& W_\tau \sim \pi_2}   \left[\int \phi df - \int \psi d\mu + \int (u -\phi) d\pi_2 +  \int \psi d\pi_1 \right]   \right\}\\
 & = \sup_{ \phi \le 0, \, \psi}  \left\{ \int \phi df - \int \psi d\mu + \inf_{\tau \le \tau^U \& \, W_0 \sim \pi_1 \& W_\tau \sim \pi_2}    \E [  u(W_\tau) -\phi(W_\tau)+  \psi (W_0)]  \right\}.
 \end{align*}
 Here the infimum in the last expression goes to $-\infty$ by choosing $\pi_1$ and $ \pi_2$ which are not necessarily probability measures, unless
 $\psi (W_0) \ge \phi (W_\tau) -  u (W_\tau) \hbox{ a.s. for any $\tau \le \tau^U$} $. Therefore, the last expression is the same as 
 \begin{align*}
 & \sup \left\{ \int \phi df - \int \psi d\mu \ \Big| \  \phi \le 0  \ \& \ \psi (W_0) \ge \phi (W_\tau) -  u (W_\tau) \hbox{ a.s. for any $\tau \le \tau^U$}   \right\}.
 \end{align*}
This gives an informal proof that $ P(\mu, f, u)= D'(\mu, f, u)
$. Below we give a rigorous proof. 

 \subsection{Proof of Theorem~\ref{th:dual}}\label{sec:dual-proof}
 Define
 \begin{align*}
 \tilde D(\mu, f, u) =  \sup_{ \phi \in C(\bar U) \ \&\  \phi \le 0 } \left[\int \phi (y) df (y)  - \int (\phi-u)^{sp} (x) d\mu(x) \right].
\end{align*}
Note that obviously $\tilde D(\mu, f, u) \le D'(\mu, f, u)$ as the pair $(\phi, (\phi -u)^{sp})$, with $0 \ge \phi \in C(\bar U)$, satisfies the  admissibility for $D'(\mu, f, u)$. 
Therefore it suffices to prove $\tilde D(\mu, f, u)  = P(\mu, f, u)$ because of the weak duality (Theorem~\ref{th:weak-dual}).

We apply  Fenchel-Rockafellar duality (see \cite[Theorem 1.9]{V1}) and  Sion's minimax theorem,  namely,\begin{theorem}[Fenchel-Rockafellar duality, Theorem 1.9 of \cite{V1}]\label{thm:F-R-duality}
Let $E$ be a normed vector space, $E^*$ its topological dual space, and $\Theta$, $\Xi$ two convex functions on $E$ with values in $\R \cup \{+\infty\}$. Let $\Theta^*$, $\Xi^*$ be the Legendre-Fenchel transforms of $\Theta$, $\Xi$ respectively. Assume that there exists $z_0 \in E$ such that 
\begin{align*}
 \Theta (z_0) < +\infty, \quad \Xi(z_0) < +\infty \quad
 \hbox{and $\Theta$ is continuous at $z_0$.}
\end{align*}
Then,
\begin{align}\label{eqn:F-R}
 \inf_E [\Theta + \Xi] = \max_{z^* \in E^*} \left[ - \Theta^* ( -z^*) - \Xi^*(z^*)\right].
\end{align}
(Here $\max$ on the righthand side is a part of the result.)
\end{theorem}

\begin{theorem} [Sion's minimax theorem \cite{Sion58}] \label{th:minmax} 
Let $X$ be a  compact convex subset of a linear topological space, and  $Y$ a convex subset of a linear topological space. Let $f : X \times Y \to \R$ be a function satisfying 
\begin{itemize}
\item [(a)] $f(x, \cdot)$ is upper semicontinuous and quasi-concave on $Y$ for each $x \in X$; 
\item[(b)] $f(\cdot, y)$ is lower semicontinuous and quasi-convex on $X$ for each $y \in Y$. 
 Then 
$$\sup_{y \in Y} \min_{x \in X}f(x,y) = \min_{x \in X}\sup_{y \in Y} f(x,y).  $$
\end{itemize} 
\end{theorem}
The rest of the proof follows a similar strategy of the proof of Kantorovich duality as in \cite{V1}. Here, our assumption that $U$ is bounded, thus its closure $\bar U$ is compact, gives a simplification: it can be relaxed but for simplicity of the paper we do not follow it. 

Let 
$ E = C (\bar U)$ be the space of all continuous functions on $\bar U$, equipped with the sup-norm $\| \cdot \|_\infty$. 
The topological dual of $E$ is the space of Radon measures on $\bar U$,  that is, $ E^* = M(\bar U)$, normed by the total variation.
We now define
 \begin{align*}
 \Theta : \phi \in C (\bar U) \mapsto
 - \int \phi (y) df (y)  + \int (\phi-u)^{sp} (x) d\mu(x)
\end{align*}
and
\begin{align*}
 \Xi : \phi \in C (\bar U) \mapsto
 \begin{cases}
    0  & \text{if $\phi \le 0$ }, \\
   +\infty   & \text{otherwise}.
\end{cases}
 \end{align*}

Let us verify the conditions of $\Theta$ and $\Xi$ for the Fenchel-Rockafellar theorem.  
The function $\Xi$ is obviously convex. For $\Theta$, notice that 
the function $\phi \mapsto (\phi-u)^{sp} = \sup_{\tau \le \tau^U} \E \left[  (\phi -u)(W^x_\tau )\right]$ is convex,  which implies $\Theta$ is convex on $C (\bar U)$.
On the other hand, consider $z_0 \equiv 0$, then, $\Theta$ is continuous at $z_0$  (with respect to sup-norm $\| \cdot\|_\infty$) and $\Theta(z_0), \Xi(z_0) <\infty$. Therefore, we can apply \eqref{eqn:F-R}.
Notice that 
\begin{align*}
 \inf_{\psi \in C (\bar U)} \left[ \Theta (\psi) + \Xi(\psi) \right]
&  = - \left\{ \sup_{\phi \in C (\bar U) \  \& \ \phi \le 0} \left[ \int \phi (y) df (y)   -\int (\phi-u)^{sp} (x) d\mu(x)\right]\right\} \\
& = - \tilde D(\mu, f, u).
\end{align*}

Next, let us consider the Legendre-Fenchel transforms $\Theta^*$, $\Xi^*$  of $\Theta$ and $\Xi$, respectively.

For each $\nu \in M(\bar U)$, we have
\begin{align*}
 \Xi ^* (\nu) & = \sup_{\phi \in C (\bar U)} \left\{  \int \phi (x)d\nu(x)- \Xi (\phi) \right\}\\
 & =\sup_{\phi \in C (\bar U) \  \& \  \phi \le 0} \left\{ \int \phi (x)d\nu(x)\right\}\\
 & = \begin{cases}
   0   & \text{if $\nu \ge 0$}, \\
   +\infty   & \text{otherwise}.
\end{cases}
\end{align*}
On the other hand,
\begin{align*}
 \Theta^* (-\nu) &=\sup_{\phi \in C (\bar U)} \left\{  -\int \phi (x)d\nu(x)- \Theta (\phi) \right\}\\
 & = \sup_{\phi \in C (\bar U)} \left\{  - \int \phi (x)d\nu(x)     + \int \phi (y) df (y) -  \int (\phi-u)^{sp} (x) d\mu(x) \right\}.
 \end{align*}
 Notice that 
 \begin{align*}
 \int (\phi-u)^{sp} (x) d\mu(x) =  \sup_{\tau \le \tau^U \ \& \  W_0\sim \mu \ \& \ W_\tau \sim \tilde \nu}\left[ \int( \phi (y) - u(y)) d\tilde \nu (y)\right].
 \end{align*}
 Therefore, 
 \begin{align}\label{eqn:Theta-inter}
  \Theta^*(-\nu) &= \sup_{\phi \in C (\bar U)} \left\{  \inf_{\tau \le \tau^U \ \& \  W_0\sim \mu \ \& \ W_\tau \sim \tilde \nu}\left[ -\int \phi (y) (d\nu +d\tilde \nu - df ) +\int  u(y) d\tilde \nu (y)\right] \right\}.
\end{align}

Here, the function $(\phi, \tilde{\nu}) \mapsto  -\int \phi (y) (d\nu +d\tilde \nu - df ) +\int  u(y) d\tilde \nu (y) $  is linear in $\phi$ and $\tilde{\nu}$, and  the set $\{\tilde{\nu} \ |  \ \tau \le \tau^U \ \& \ W_0\sim \mu  \ \& \ W_\tau \sim \tilde \nu \}$ 
  is compact with respect to the weak*-topology as a closed set of probability measures on the compact set $\bar U$. (Notice that the set of randomized stopping times, as  probability measures, is endowed with the weak*-topology, that is,  the weak* topology is its original topology.)

 Hence the infimum is attained at some $\tilde{\nu}$ for a given $\phi$, and we can apply Theorem~\ref{th:minmax} to get
\begin{align*}
 \Theta^* (-\nu) &= \sup_{\phi \in C (\bar U)} \left\{  \inf_{\tau \le \tau^U \ \& \  W_0\sim \mu \ \& \ W_\tau \sim \tilde \nu}\left[ -\int \phi (y) (d\nu +d\tilde \nu - df ) +\int  u(y) d\tilde \nu (y)\right] \right\}
 \\ & = \sup_{\phi \in C (\bar U)} \left\{  \min_{\tau \le \tau^U \ \& \  W_0\sim \mu \ \& \ W_\tau \sim \tilde \nu}\left[ -\int \phi (y) (d\nu +d\tilde \nu - df ) +\int  u(y) d\tilde \nu (y)\right] \right\}
\\ & = \min_{\tau \le \tau^U \ \& \  W_0\sim \mu \ \& \ W_\tau \sim \tilde \nu} \left\{ \sup_{\phi \in C (\bar U)} \left[ -\int \phi (y) (d\nu +d\tilde \nu - df ) +\int  u(y) d\tilde \nu (y)\right] \right\}.
 \end{align*}
Observe that $\int u(y) d\tilde \nu$ is bounded by $\|u\|_\infty  \tilde \nu (\bar U)=\|u\|_\infty \mu (U)$. Therefore, 
\begin{align*}
  \Theta^*(-\nu) \ge \min_{\tau \le \tau^U \ \& \  W_0\sim \mu \ \& \ W_\tau \sim \tilde \nu} \left\{ \sup_{\phi \in C (\bar U)}\left[ -\int \phi (y) (d\nu +d\tilde \nu - df ) \right] \right\} - \|u\|_\infty \mu (U).
\end{align*}

 Given $\nu$,  suppose that
$\tilde{\nu} + \nu \neq f$ on a set with positive measure. Then for some $\phi\in C (\bar U)$ we have
   $$ - \int \phi (y) (d\nu +d\tilde \nu - df )  \not =0.$$ 
 Then, by multiplying $\phi$ by a constant $\lambda$ (with the opposite sign if necessary),   we can make $$ -\int \lambda \phi (y) (d\nu +d\tilde \nu - df )  \longrightarrow \infty \quad \hbox{ as $|\lambda| \to \infty$.}
 $$
 Therefore, the above supremum in \eqref{eqn:Theta-inter} becomes
\begin{align*}
\Theta^*(-\nu) 
 = \min_{\tau\le \tau^U \ \& \, W_0\sim \mu \  \& \, W_\tau \sim \tilde \nu \hbox{ and } \tilde \nu + \nu =f}   \int u (y) d\tilde \nu (y)
\end{align*}
where it is understood that this minimum is $+\infty$ if there is no such $\tilde \nu$ with the condition ``$\tau\le \tau^U \ \& \, W_0\sim \mu\  \& \, W_\tau \sim \tilde \nu \hbox{ and } \tilde \nu + \nu =f$.''
Then together with $\Xi^*$, this implies 
\begin{align*}
& \max_{z^* \in E^*} \left[ - \Theta^* ( -z^*) - \Xi^*(z^*)\right]\\
& = - \min \left\{ \int u(y) d\tilde \nu (y)  \,\, : 
 \hbox{ there exist $\tilde \nu \le f$, $\tau$  
  such that  $\tau \le \tau^U \ \& \  W_0\sim \mu \ \& \ W_\tau \sim \tilde \nu$} \right\}\\
&= - P(\mu, f, u).
\end{align*}

This completes the proof of Theorem~\ref{th:dual}. 

\section{Attainment of the dual problem}\label{sec:dual-attainment}
In this section, we prove the dual attainment using the following lemma and proposition for superharnomic functions.

\begin{lemma}[Local Maximum Principle (c.f. Theorem 9.20 in \cite{GT-book})]\label{lem:localMax}
Let $U \subset \R^d$ be a  bounded open set.  
Let $\psi:U\to \R$ satisfy $\Delta \psi \le 0$ and $\psi \le 0$. Then for any ball $B_r$ such that $2B_r \subset U$ and for any $p >0$, there exists a constant $C = C(p, r, U)$ such that 
\begin{align*}
 \sup_{B_r} |\psi| \le C \frac{1}{|2B_{r}|}\left( \int_{2B_r} |\psi|^p \right)^{1/p}
\end{align*}
\end{lemma}
\begin{proof}
 This is a special case of Theorem 9.20 of  \cite{GT-book}.
\end{proof}
Lemma~\ref{lem:localMax} also holds for LSC nonsmooth superharmonic functions due to the standard mollifier process: we can approximate $\phi$ by smooth superharmonic functions, using a convolution with a smooth mollifier. This lemma then provides the following proposition.

\begin{proposition}\label{prop:limit-sp}
 
Let $U \subset \R^d$ be a  bounded open set.  
Let $\psi$ and $\psi_j$ be  measurable LSC functions on $U$ such that 

\begin{itemize}
 \item 
$ \psi_j \le 0$ in $U$;
\item  $\psi_j$’s are superharmonic   in $U$;
\item $ \int_U |\psi_i| dx \le C $ for some universal constant $C$ (independent of $j$);
\item   $\psi_j \to \psi$ as $j\to \infty,$ for a.e. $x \in U$.

\end{itemize}
Then, $\psi$ is superharmonic in $U$.
\end{proposition}

\begin{proof}
 From Lemma~\ref{lem:localMax} applied with $p=1$ we have 
that  in every compact ball in $U$ there is a uniform lower bound on $\psi_j$'s; the lower bound may depend on the compact ball, but not on $j$.
Therefore we can apply Fatou's lemma to show the desired  property of  Definiton~\ref{def:conventional-superharmonic}
 for $\psi$ for a.e. $x \in U$. Then since $\psi$ is LSC, the  property of Definition~\ref{def:conventional-superharmonic} holds for every $x \in U$. 
 \end{proof}

We now prove the dual attainment, using the above results. 
\begin{theorem}[Dual attainment]\label{th:dual-attain}

 For $\mu$, $f$, $u$ given as in Section~\ref{sec:Assumptions and notation}, let $\mu$ satisfy \eqref{eqn:SHmu}. 
Then, the dual problem $D(\mu, f, u) $ is  attained with an optimal solution $( \phi^*, \psi^*) \in \mathcal{D}$. 
\end{theorem}

\begin{proof}  We prove the theorem in five steps.  Recall $D(\mu, f,u) = D'(\mu, f, u) = P(\mu, f, u)$.

\medskip

{\bf Step 1. }
Let $\{(\phi_i, \psi_i)\} \subset \mathcal{D}'$ be a maximizing sequence for $D'(\mu, f, u)$, in particular, $\phi_i \le 0$.
Because $\tilde D(\mu, f, u) = D'(\mu, f, u)$  from  the proof of Theorem~\ref{sec:dual-proof},   we may assume that $\phi_i \in C(\bar U)$ and $\psi_i = (\phi_i - u)^{sp}$.

Now we increase $\phi_i\le 0$ point-wise as much as possible without changing $\psi_i$,  which is favored for the maximization of the functional $\phi \to \int \phi df - \int \psi d\mu$. Then since $\psi_i = (\phi_i-u)^{sp}$ is the smallest superharmonic function $\ge \phi_i-u$, the largest $\phi_i$ (holding  the same  $\psi_i$)  satisfies
\begin{itemize} 
 \item[(a)]  $\phi_i(x) =0$ (where $0$ is the upper bound of $\phi_i$), if $(\phi_i -u)^{sp} (x) > -u (x)$;
  \item[(b)]  $(\phi_i -u) (x)= (\phi_i - u)^{sp}(x)$,  otherwise. 

\end{itemize}
Hence we may assume without loss of generality that
$$\phi_i- u = \min[ (\phi_i -u)^{sp}, -u].$$
From this we see that 
 \begin{align*}
 0 &\le (\phi_i - u)^{sp}- (\phi_i - u)\\
 &   = [(\phi_i - u)^{sp} - (-u)]\chi_{\{(\phi_i - u)^{sp} > -u\}} \\ 
 & \le [(\phi_i - u)^{sp} - (-u)]\chi_{\{\phi_i =0\}} \\ 
 &\le  u\chi_{\{\phi_i =0\}} \quad \\
 & \le u
\end{align*}
where the third inequality follows since  $(\phi_i -u )^{sp} \le 0$ (here recall that $(\phi_i -u )^{sp}$ is the smallest superharmonic function above  $\phi_i -u $, while $\phi_i -u \leq 0$).
For simplicity of notation, we define 
\begin{align}\label{eqn:phi-hat}
 \hat \phi_i := (\phi_i -u)^{sp} - ( \phi_i - u).
\end{align}
Then from the above estimate, we have
\begin{align*}
 0\le \hat \phi_i \le u \qquad \hbox{and} \quad \int |\hat \phi_i| \le C.
\end{align*}
\medskip

{\bf Step 2.} 
We will show that 
\begin{align}\label{eqn:phi-i-bound}
 \int| \phi_i | dx \le C' \hbox{ for all large } i 
\end{align}
where $C'$ is a constant independent of $i$ (but  dependent  on $\delta$ among others). 
To see this first note 
 $$ \int \phi_i df - \int (\phi_i -u)^{sp} d\mu  \rightarrow D'(\mu, f,  u) = D(\mu, f, u). $$
 Recall that by our assumption \eqref{eqn:SHmu} there is a measure $\nu$ with $\nu \leq \delta f \chi_U$, $0< \delta <1$, and a stopping time $\tau\le \tau_U$ such that  $W_0\sim\mu, W_\tau\sim\nu$. With this $\nu$ we observe that 
 \begin{align}\nonumber
  & \int \phi_i df - \int (\phi_i -u)^{sp} d\mu\\\nonumber &=
\int u df +   \int (\phi_i -u)df  - \int (\phi_i -u)^{sp} d\mu \\ \label{eqn:with-nu}
& = \int u df +   \int (\phi_i -u)(df-d\nu) + \int  (\phi_i -u)d\nu  - \int (\phi_i -u)^{sp} d\mu \\ \nonumber
& \le  \int u df +   \int (\phi_i -u)(df-d\nu) +\left[ \int  (\phi_i -u)^{sp}d\nu  - \int (\phi_i -u)^{sp} d\mu \right] \\ \nonumber
& \qquad \qquad  \hbox{ (since $(\phi_i -u)^{sp} \ge (\phi_i -u)$)} \\ \nonumber
&  \le  \int u df +   \int (\phi_i -u)(df-d\nu)  \hbox{ (from  \eqref{eqn:maximum-good})}\\ \nonumber
& =\int ud\nu+ \int \phi_i (df -d\nu)\\ \nonumber
& \le \int udf+ \int \phi_i (df -d\nu).
  \end{align}
Since $\int udf$ and $D(\mu, f, u)$ are finite,  $\phi_i \le 0$ and $f-\nu \ge 0$, it follows that
$\int | \phi_i|  (df -d\nu) \le C$ for a constant $C$ independent of $i$. Since $df -d\nu \ge (1-\delta) f dx$, this shows 
 \eqref{eqn:phi-i-bound} as desired. 

\medskip

 {\bf Step 3.}
From the above $L^1$ bounds for $\phi_i$ and $\hat \phi_i$ obtained in steps 1 and 2,  we will now construct a.e. convergent sequences $\phi^*_i $ and  $ \tilde{ \phi^*}_i$, where $\phi^*_i$ is a maximizing sequence of $D(\mu, f, u)$. Indeed, from Koml\'os Theorem~\ref{thm:A} and the bounds obtained in steps 1 and 2 we have a subsequence $\{ \phi'_k\} = \{ \phi_{i_k}\}$ and $\phi^*$ such that for every arbitrary  subsequence $\{ \phi'_{k_n} \}$ if $n\to \infty$ then 
\begin{align*}
 \frac{1}{n} (\phi'_{k_1}+ ... + \phi'_{k_n}) \to \phi^* \hbox{ a.e.} 
\end{align*}
Then apply Koml\'os theorem  to  $\hat{\phi'}_k$ (defined as  \eqref{eqn:phi-hat}) to get a subsequence $\hat{\phi'}_{k_j}$ and $\tilde \phi^*$ having the above convergence property to $\tilde \phi^*$ for any subsequences.
Now, use the index of  $\hat{\phi'}_{k_j}$ to both $\phi'_{k_j} $ and $\hat{\phi'}_{k_j}$, to define 
\begin{align*}
\phi^*_n :=\frac{1}{n} \sum_{j=1}^{n} \phi'_{k_j}, \quad \tilde{\phi^*}_n :=\frac{1}{n} \sum_{j=1}^{n} \hat{\phi'}_{k_j} .
\end{align*}
Note that we have
$$\phi^*_n \to \phi^*, \quad  \tilde{ \phi^*}_n \to \tilde{\phi^*}
\quad \hbox{ a.e.}$$   Furthermore,  $\tilde{\phi^*}_n$ inherits the following estimates from $\hat{\phi}_i$: 
\begin{align*}
0\le \tilde{\phi^*}_n \le u \hbox{ for all $n$}  \quad \& \quad   \hbox{ $\int|\phi^*_n - u|  \le C'$ for all large   $n$}.
\end{align*}

It remains to show that $\phi_i^*$ is a maximizing sequence. To see this first observe  
\begin{align*}
 (\phi^*_n -u)^{sp}  =\left(  \frac{1}{n} \sum_{j=1}^n    \phi'_{k_j}    - u\right)^{sp} \, \le \,   \frac{1}{n} \sum_{j=1}^n   \left( \phi'_{k_j}    - u\right)^{sp}.
\end{align*}
Therefore, 
\begin{align*}
 & \int \phi^*_n df - \int (\phi^*_n -u)^{sp} d\mu\\
& \ge  \frac{1}{n} \sum_{j=1}^n\left[ \int \phi_{k_j} df - \int (\phi_{k_j} -u)^{sp} d\mu\right]\\
 & \to  D(\mu, f, u),
 \end{align*}
where the last convergence follows since $ \int \phi_{k_j} df - \int (\phi_{k_j} -u)^{sp} d\mu$ converges to $D(\mu, f, u)$ as $j\to \infty$.
This completes step 3.

\medskip

 {\bf Step 4.} 
 We now show that $\tilde{\phi^*} + (\phi^* -u)$ is superharmonic and $\tilde{\phi^*} + (\phi^* -u) \ge  (\phi^* -u)^{sp} $.
Indeed  \eqref{eqn:phi-hat}  yields
\begin{align*}
  \tilde{ \phi^*}_n + (\phi^*_n-u) = \frac{1}{n} \sum_{j=1}^n ( \phi'_{k_j}  -u +\hat{\phi'}_{k_j} )=  \frac{1}{n} \sum_{j=1}^n (\phi'_{k_j}-u)^{sp},
\end{align*}
where the latter is superharmonic, and has uniformly bounded $L^1$ integral on $U$.  This is because 
both $\hat{\phi_i}$ and  $\phi_i$ have uniform $L^1$ bounds. By Lemma~\ref{lem:LSC}, $(\phi_{k_j}-u)^{sp}$ is LSC for each $j$. Moreover, 
\begin{align*}
 \tilde{\phi^*}_n + (\phi^*_n -u)
 \rightarrow \tilde{ \phi^*} + (\phi^* -u)\quad  \hbox{ a.e.}
 \end{align*}
Therefore from Proposition~\ref{prop:limit-sp}, we see that 
$  \psi^* := \tilde{\phi^*} + (\phi^* -u)$
 is superharmonic in the sense of  Definition~\ref{def:conventional-superharmonic}.  Since $\tilde{ \phi^*} \ge 0$, we see 
$$ \psi^*  \ge (\phi^* -u).
$$ 

 Also, $ \phi^* \le 0$ as the limit of $\phi^* _k\le 0$, therefore, the pair $(\phi^*,  \psi^*) = ( \phi^*, \tilde{\phi^*} + (\phi^* -u)) \in \mathcal{D}$.

\medskip

{\bf Step 5.} 
 This is the final step where we verify optimality of $(\phi^*, \psi^*)$. Recall from Step 2 that $\tau$ satisfies $W_0\sim \mu$, $W_\tau \sim \nu$ and $\nu$ satisfies  $\nu\le \delta f\chi_U$.
Recall also from Step 3 that $\phi^*_k$ is a maximizing sequence  of $D(\mu, f, u)$, namely
\begin{align}\label{eqn:step5}
  D(\mu, f, u)
   &  =\lim_k \left[ \int \phi^*_k df - \int (\phi^*_k -u)^{sp} d\mu\right].
   \end{align}
The last expression inside of the limit can be written, similarly to \eqref{eqn:with-nu}, as
   \begin{align*}
\int u df +   \int (\phi^*_k -u)(df-d\nu) + \int  (\phi^*_k -u)d\nu  - \int (\phi^*_k -u)^{sp} d\mu. 
\end{align*}
We can write this with $\tau$ as 
$$
  \int u df +   \underbrace{\int (\phi^*_k -u)(df-d \nu) }_{\text{$:=I$}}+ \underbrace{\int  \left (\mathbb{E}[\phi^*_k (W^x_{\tau})-u(W^x_{ \tau})] - (\phi^*_k -u)^{sp}(x) \right )d\mu(x)}_{\text{$:=II$}}. 
 $$
Recall that 
$\phi^*_k -u \le 0$ and $f -\nu \ge0$ by construction,  therefore applying Fatou's lemma we get 
\begin{align*}
 \limsup_k I \le  \int (\phi^* -u)(df-d \nu).
\end{align*}
We now estimate the $\limsup$ of $II$.
Notice that $\tau \le \tau^U$, and thus \eqref{eq:sp} yields  
$$\mathbb{E}[(\phi^*_k (W^x_{\tau})-u(W^x_{ \tau})]  - (\phi^*_k -u)^{sp}(x) \le 0.$$ Thus due to Fatou's lemma
\begin{align*}
\limsup_k II  & \le \int  \limsup_k \Big( \mathbb{E}[(\phi^*_k (W^x_{\tau})-u(W^x_{\tau})]  - (\phi^*_k -u)^{sp}(x)
\Big) d\mu(x)\\
& \le \underbrace{\int  \limsup_k \Big( \mathbb{E}[(\phi^*_k (W^x_{\tau})-u(W^x_{\tau})]  d\mu(x)\Big) d\mu(x)}_{\text{$:=III$}} + \underbrace{\int (- \liminf_k (\phi^*_k -u)^{sp}(x) )d\mu(x)}_{\text{$:=IV$}}.
 \end{align*}
 
For $III$, since $\phi^*_k -u \le 0$, Fatou's lemma again yields
\begin{align*}
\limsup_k \mathbb{E}[(\phi^*_k (W^x_{\tau})-u(W^x_{\tau})]  \le \mathbb{E}\left[\limsup_k (\phi^*_k (W^x_{\tau})-u(W^x_{\tau}))\right] \hbox{ for each $x$}.
\end{align*}
Thus, 
\begin{align*}
 III & \le \int  \mathbb{E}\left[\limsup_k (\phi^*_k (W^x_{\tau})-u(W^x_{\tau}))\right]  d\mu(x) \\
 & = \int \limsup_k (\phi^*_k - u) d\nu  \quad \hbox{ (as $W_\tau \sim \nu$ and  $\limsup_k (\phi^*_k - u) \le 0$)}\\
 & = \int  (\phi^* - u) d\nu \quad \hbox{ (as $\phi^*_k \to \phi^*$ a.e.  and $\nu \ll Leb$). }
\end{align*}

For $IV$, $(\phi^*_k -u)^{sp} = ( \tilde{\phi^*}_k + \phi^*_k -u)$ (by the definition of $\tilde{\phi^*}_k$), thus, 
\begin{align*}
 \liminf_k \left( (\phi^*_k -u)^{sp}\right) = \liminf_k ( \tilde{\phi^*}_k + \phi^*_k -u).
\end{align*}
Since $\phi^*_k , \tilde \phi^*_k \to \phi^*, \tilde \phi^*$ a.e and $\mu \ll Leb$, it follows that
\begin{align*}
IV = - \int (\tilde \phi^* + \phi^* -u )d\mu.
\end{align*}

Above estimates for $III, IV$ yield
\begin{align*}
 \limsup_k II \le  \int  (\phi^* - u) d\nu - \int (\tilde \phi^* + \phi^* -u )d\mu.
\end{align*}
Finally going back to \eqref{eqn:step5},  above estimates on $I, II$ and the fact $\psi^* =  \tilde \phi^* + \phi^* -u$ yield
\begin{align*}
 D(\mu, f, u) & \le  \int u df +  \int (\phi^* -u)(df-d\nu)+
 \int ( \phi^*-u) d\nu  -\int \psi^* d\mu \\
 & = \int \phi^*df  -\int \psi^* d\mu .
 \end{align*}
 On the other hand, from Step 4 the pair $(\phi^*, \psi^* )$ is in  $\mathcal{D}$, and thus is admissible for $D(\mu, f, u)$.  Hence, 
the above line is less than or equal to $D(\mu, f, u)$  and we can conclude.
\end{proof}

\begin{remark}\label{rmk:delta}
The  proof of dual attainment (Theorem~\ref{th:dual-attain}), more precisely Step 2 therein, is the only place where we need $\delta$ in $(C_0)$ to be strictly less than $1$. 
\end{remark}

\section{Saturation property for the primal optimal solution}\label{sec:saturation}

We prove in this section that the optimal target measure $\nu^*$ of $P(\mu, f, u)$  saturates the density constraint wherever it is positive. 
To motivate the statement of the main result of this section (Theorem~\ref{th:saturation}) and also  to demonstrate why the optimality of $\nu^*$ is related to its saturation, we first give the following (easy) proposition before proving the main theorem: this proposition will not be used anywhere else in the paper.  
\begin{proposition}\label{prop:easy-sat}
Let $\nu^*$ be an optimal solution to $P(\mu, f, u)$. 
Then, for each $\epsilon >0$ there is no open set $V \subset \R^d$ such that 
\begin{align*}
0<  \nu^*(x) < f (x) -\epsilon \quad \hbox{ for a.e. $x \in V$.}
\end{align*}

\end{proposition}
\begin{proof}
 
 Let $\tau^*$ denote the optimal stopping time (via \cite{beiglboeck2017optimal, PDE19}) with $W_{\tau^*} \sim \nu^*$ and $W_0 \sim \mu$. 
 Suppose by contradiction that there is $\epsilon>0$ and an open set $V$ with $0< \nu^*(x) < f(x) -\epsilon$ for a.e. $x \in V$. 
Pick $x \in V$, and consider small $r>0$ such that $B_{2r}(x) \subset V$. Then, there exists a stopping time $\tau$ with initial distribution $\chi_{B_r(x)}$ and final distribution $\delta \chi_{B_{2r}(x)}$ for some dimensional constant $\delta>0$. 
Then because of the assumption on $V$, there exists $\delta_1 >0$, depending on $\epsilon$ and $r$,  such that the restriction of $\tau$, say, $\tau_1$, to the initial distribution $\delta_1 \nu^*|_{B_r(x)}$ has final distribution, say, $\nu_1$ on $B_{2r}(x)$, with $\nu_1 \le (f - \nu^*)|_{B_{2r}(x)}$. 
We extend the stopping time $\tau_1$ to be $0$ for the initial distribution $\nu^* - \delta_1 \nu^*|_{B_r(x)}$. 
In other words, 

\begin{align*}
\tau_1 :=
 \begin{cases}
  \tau_1  & \text{for Brownian motion starting from the distribution $\delta_1 \nu^*|_{B_{r}(x)}$ }, \\
   0   & \text{otherwise, i.e. for  Brownian motion starting from the distribution  $\nu^* - \delta_1 \nu^*|_{B_r(x)}$}.
\end{cases}
\end{align*}
Let $\nu_1^*$ denote the corresponding final distribution $W_{\tau_1} \sim \nu_1^*$  with $W_{0} \sim \nu^*$ .  Because of the construction we have $\nu^*_1 \le f$. 
Notice that for $\nu^*_1$, there is a  corresponding stopping time from the initial distribution $\mu$, obtained by concatenating the stopping time $\tau^*$ for $\nu^*$ and $\tau_1$. So $\nu^*_1$  is admissible for the problem $P(\mu, f, u)$ (see \eqref{eqn:primal-problem}). Also, notice that  $Prob[ \tau^*_1 >0] >0$.
Since $u$ is strictly superharmonic ($\Delta u <0$), this then implies 
\begin{align*}
 \int u d\nu^*_1 < \int u d\nu^*,
\end{align*}
which contradicts with the optimality of $\nu^*$ for $P(\mu, f, u)$. This completes the proof. 
\end{proof}

\begin{remark}\label{rmk:easy-saturation}
Proposition~\ref{prop:easy-sat} gives a  version of the saturation result on open sets. It therefore would give the main result below (Theorem~\ref{th:saturation}), assuming $f$ is continuous, if the positive set of $\nu^*$ were open a.e.,  which is not known to us. 
 We expect it to be false given that, formally, our target measure $\nu^*$ describes the shape of ice crystals generated by supercooling process.
\end{remark}

Before we present the main theorem of this section, let us state  the  following remarkable result proven in \cite{Frank-Lieb18}:
\begin{lemma}[Theorem 16 of \cite{Frank-Lieb18}]\label{lemm:no-flat}
Let $U \subset \R^d$ be open  and $\psi \in L^1_{loc}(U)$ be strictly subharmonic in the sense that $\int_E \Delta \psi >0$  for any measurable $E \subset U$ with $|E| >0$. Then $| \{ \psi =r\}| =0$ for any $r \in \R$.
(Here $\int_E \Delta \psi$ should be understood as $\mu_\psi (E)$ for the Borel measure $\mu_\psi$ which is generated by the nonnegative distribution $\Delta \psi$.)
\end{lemma}
From this we immediately deduce the following proposition that we need:
\begin{proposition}\label{prop:the-same-set}
 Let $u_1, u_2: U \to \R$ be two functions such that $u_1$ is superharmonic as in \eqref{eqn:superharmonic} with $u_1 \in L^1_{loc}(U)$
and that  $u_2 \in C^2(U)$ and $\Delta u_2 >0$ in $U$. Then, 
\begin{align*}
 | \{  x \in U \ | \ u_1 (x)= u_2 (x) \} | =0.
\end{align*}
\end{proposition}
\begin{proof}
Note that $u_2 - u_1$ satisfies the condition in Lemma~\ref{lemm:no-flat}. Therefore the result follows.
\end{proof}

We now prove the main theorem of this section, whose proof crucially uses the dual attainment (Theorem~\ref{th:dual-attain})   as well as Proposition~\ref{prop:the-same-set}.

\begin{theorem}[Saturation]\label{th:saturation}
 
Let $\mu$ be as given in Theorem~\ref{th:dual-attain}, let  $\nu^*$ be an optimal measure for $P(\mu, f, u)$ and $(\phi^*, \psi^*)$ an optimal solution to $D(\mu, f, u)$.
 Then
 $$\nu^* = f \chi_{\{ \phi^* <0\} }.$$

\end{theorem}

\begin{proof}
Consider an optimal primal/dual pair $(\nu^*, (\phi^*, \psi^*))$ for $P(\mu, f, u)$ and $D(\mu, f, u)$ from Theorem~\ref{th:dual-attain}, respectively. Recall from Theorem~\ref{th:dual-attain} that $\psi^*$ is chosen to satisfy Definition~\ref{def:conventional-superharmonic} so satisfies \eqref{eqn:superharmonic}  with  $\psi^* \ge \phi^* -u$. 
 Let  $\tau^*$ be the optimal stopping time with $W_0 \sim \mu$ and $W_{\tau^*} \sim \nu^*$.  Then, proceeding exactly as in the proof of  Theorem~\ref{th:weak-dual} using  Lemma~\ref{lem:maximum-principle-stopping} and Lemma~\ref{lem:stopping-lemmas}, we see that 
 \begin{equation}\label{observation_1-3}
 \int \psi^* d\mu  \ge \int \psi^* d\nu^* \ge  
  \int  (\phi^* -u)  d\nu^*.
\end{equation}
Then from Theorem~\ref{th:dual} we see that (as in the proof of  Theorem~\ref{th:weak-dual}) 
\begin{align*}
 \int u d\nu^* = P(\mu, f, u) &= D(\mu, f,  u)  = \int \phi^* df - \int \psi^* d\mu\\
 & \le \int \phi^*  df- \int (\phi^* -u)d\nu^*\\  & = \int \phi^*(df-d\nu^*)  +\int ud\nu^* \\
 &\le \int u d\nu^* \hbox{ (since $\phi^*\le 0$ and $\nu^*\le f$)}.
\end{align*}

\medskip

It follows that we have
\begin{itemize}
 \item[$(i)$] $ \int\psi^* d\mu  =  \int (\phi^* -u) d\nu^*  $ and \\
 \item[$(ii)$]  $ \int \phi^* (df -d\nu^*) =0$.
\end{itemize}
From $(ii)$ and the fact that  $\phi^* \le 0$ and $ \nu^*\leq f$, the following is immediate:
$$ \nu^* =f  \hbox{ a.e. wherever }\phi^* <0. $$

\medskip

Next we claim that  $\nu^*=0$ wherever $\phi^* =0$, namely  
\begin{equation}\label{claim2}
 \nu^*[\{\phi^* =0\}] = 0,
\end{equation} 
which will yield the desired saturation property, $\nu^* = f \chi_{\{ \phi^* <0\} }$.

\medskip

The rest of the proof presents the proof of \eqref{claim2}. 
 First  from $(i)$,  \eqref{observation_1-3} and the fact that $\psi^* \ge \phi^*-u$, it follows
\begin{align}\label{eqn:sp-strict}
 \nu^* [ \{ x \ | \  \psi^* (x) > (\phi^* -u)(x)\}]=0.
\end{align}

\medskip

 To show \eqref{claim2}
using  \eqref{eqn:sp-strict}, it is sufficient  to prove:
\begin{align*}
| \{ \phi^* =0\} \setminus \{ \psi^*  > (\phi^* -u)\}|=0. 
\end{align*}
Since $\psi^* \geq \phi^* - u$, the above set is the same as $\{\phi=0\} \cap \{\psi^* = \phi^*-u\}$. Notice that 
\begin{align*}
\{ \phi^* =0 \ \& \ \psi^*  = (\phi^* -u) \} 
 =  \{ \psi^*  = (\phi^* -u) =-u \} .
\end{align*}

Now, $\psi^*$ is superharmonic as Definition~\ref{def:conventional-superharmonic}   and $-u$ is  smooth and  strictly subharmonic, therefore by Proposition~\ref{prop:the-same-set},  the set  $\{ \psi^*  =-u \} $ has zero Lebesgue measure. This completes the proof of \eqref{claim2} and we conclude.
\end{proof}

The above saturation result gives a characterization of the optimal target $\nu^*$ as well as its uniqueness. 
\begin{corollary}[Uniqueness with fixed $u$]\label{cor:uniqueness-primal}
For $\mu$ as given in Theorem~\ref{th:dual-attain}, the primal problem $P(\mu, f, u)$
has a unique optimal solution.
\end{corollary}
\begin{proof}
 This follows from Theorem~\ref{th:saturation} because $\nu^* = f \chi_{\{\phi^* <0\}}$ is completely determined by any optimal dual solution $(\phi^*, \psi^*)$. This also shows that the set $\{ \phi^* <0\}$, equivalently $\{ \phi^* =0\}$, is determined $f$-a.e. by the primal solution $\nu^*$. 
\end{proof}

It is not clear whether the optimal $\nu^*$ can even be determined independent of the choice of $u$.
Also, at this moment it is not clear whether the dual problem $D(\mu, f, u)$ has a unique optimal solution, though all its optimal solutions have the same zero set (of $\phi^*$, up to $f$-measure zero set).

\section{Existence for the supercooled Stefan  problem}\label{sec:main}
In this section we focus on the case $f\equiv 1$ to prove our main theorem, Theorem~\ref{thm:A},  for the supercooled Stefan problem.  The results in the  previous sections together with \cite{Kim-Kim21} will give us the existence of a weak solution of $(St)$, with the initial data $\eta_0$ and initial domain $\{\eta_0>0\}$. 

\medskip

Let us recall the definition of a weak solution to the supercooled Stefan problems with initial density $\eta_0\in L^1(\R^d)$ and initial domain $E\subset \R^d$:
\begin{equation*} 
(\eta -  \chi_{\{\eta>0\}})_t - \frac{1}{2}\Delta\eta = 0,  \qquad \eta(\cdot,0) =\eta_0 \in L^1(\R^d), \quad E:=\limsup_{t\to0^+}\{\eta(\cdot,t)>0\}. \leqno(St)
\end{equation*}

\begin{definition}[See e.g. \cite{Kim-Kim21}]\label{def:weak-solution}
A nonnegative function $\eta\in L^1(\R^d\times [0,\infty))$ is a weak solution of $(St)$ with initial data $(\eta_0, E)$  with $\eta_0\in L^1(\R^d)$  and  $\{\eta_0>0\}\subset E$ a.e.  if
\begin{itemize}
\item[(a)] the set $\{\eta(\cdot,t)>0\}$ decreases  in $t$;  
\item[(b)] $E= \limsup_{t\to0^+} \{\eta(\cdot,t)>0\}$  ; \\
\item[(c)] for any test function $\varphi\in C^{\infty}_c(\R^d\times [0,\infty))$,
 \begin{equation}\label{weak:stefan}
\int_0^{\infty}\int_{\R^d} [(\eta -   \chi_{\{\eta>0\}})\varphi_t +\frac{1}{2}\eta\Delta\varphi] dx dt =  \int_{\R^d} [(-\eta_0 + \chi_E)\varphi ](x,0) dx. 
\end{equation} 
\end{itemize}

\end{definition}

 We also point out that the condition $(C)$ given in Section~\ref{sec:intro} is necessary in the existence of $(St)$ (as in Definition~\ref{def:weak-solution}) with initial domain  $\{\eta_0>0\}$. 

\begin{theorem}\label{thm:necessary}
Suppose $\eta$ is a weak solution of $(St)$ with initial data $(\eta_0, E=\{\eta_0>0\})$. Then $\eta_0$ satisfies the condition (C), namely, it is strictly subharmonically ordered to $\nu:= \chi_A$  with respect to $E$, where $A$ is a subset of $E$.
\end{theorem}

\begin{proof}
From  \cite[Theorem 5.6]{Kim-Kim21} it follows that $\eta_0$ is subharmonically ordered to $\nu\chi_{\{s(x)<\infty\}}$ for a function $s(x)$. Moreover the weak solution $\eta(\cdot,t)$ represents the active distribution of heat particles at time $t$.  One can thus deduce that there is no instantly stopped particles, namely $\tau>0$ a.e., since one can easily check from the weak solution definition that $\lim_{t\to 0} \eta(\cdot,t)= \eta_0$ a.e.
\end{proof} 

\begin{remark}\label{rem:zone}
When $E =\{ \eta_0 >0\}$, the set $A$ in the above theorem is nothing but the transition zone $\Sigma(\eta)$ defined in the introduction.
 \end{remark}

We now prove Theorem~\ref{thm:A}, which is restated here in the context of previous sections.
\begin{theorem}[Global existence]\label{th:Stefan-weak}
For $\mu$ and $U$ as given in Section~\ref{sec:prelim}, let us denote $\eta_0$ as the density of $\mu$. Suppose that $\{\eta_0>0\}=U$.
Further suppose that $\eta_0$ satisfies $(C_0)$ in Definition~\ref{def:C0} with $|\{\eta_0=1\}|=0$.

Then there exists a  weak solution $\eta$ to $(St)$ with the initial data $\eta_0$ and the initial domain $E=U$ a.e.,  corresponding to the Eulerian variable  associated to the optimal solution $\nu^*$ to  $P(\mu, f=1, u)$. Moreover, $\eta$ is maximal in the sense of Definition~\ref{def:maximal}.
\end{theorem}

\begin{proof}
   In the following proof we use the results of \cite{PDE19, Kim-Kim21}.  Let $u$ be a superharmonic function and let $\nu^*$ be the optimal solution (target distribution) to $P(\mu, f=1, u)$ as given in  Theorem~\ref{th:primal-existene}. 
 Then we have $\nu^* =  \chi_{\{ \phi^* <0\} }$ from Theorem~\ref{th:saturation} for a dual optimal solution $(\phi^*, \psi^*)\in\mathcal{D}$ given as in  Theorem~\ref{th:dual-attain}:  note that $\mu$ satisfies \eqref{eqn:SHmu} as it satisfies the stronger condition $(C_0)$.
 Let  $\tau^*$ be the optimal stopping time  for $\nu^*$ obtained via the optimal Skorokhod problem for a type (I) cost as in \cite{PDE19}. 
  Note that the value of $P(\mu, f=1, u)$ depends only on the optimal target distribution $\nu^*$: thus it follows that  the pair $(\tau^*, \nu^*)$ is an optimal solution of $P(\mu, f=1, u)$. 
  
 Since $|\{ \eta_0=1\} | =0$ by our assumption, and since the density of $\nu^*$ is a characteristic function,  $\tau^*>0$ almost surely.
 This can be seen from \cite[Section 8]{Kim-Kim21}. 
Moreover from \cite[Section~\ref{sec:dual-attainment}]{Kim-Kim21} we have the corresponding Eulerian flow $(\eta, \rho)$ as given in \cite{PDE19} (or see \cite[Section~\ref{sec:prelim}]{Kim-Kim21}) yielding a solution to the weighted supercooled problem, namely, 
 \begin{equation}\label{Stefan-nu}
(\eta -  \nu^*\chi_{\{\eta>0\}})_t - \frac{1}{2}\Delta\eta = 0,  \qquad \eta(\cdot,0) =\eta_0 \in L^1(\R^d), \quad E:=\limsup_{t\to0^+}\{\eta(\cdot,t)>0\}.
\end{equation}
We mention that $E=U$ a.e.  follows from the fact  $\tau^*>0$ and the assumption that   $\eta_0>0$ on $U$.
It is important to note that the equation \eqref{Stefan-nu} depends only on the value of $\nu^*$ on the set $\{ x :   (x, t) \in \partial\{ \eta >0\}\}$,  in which the free boundary goes through for some time, on the support of the final distribution $\nu^*$ of $\tau^*$: see \cite[Remark 5.2]{Kim-Kim21}. 
Because $\nu^* =  \chi_{\{ \phi^* <0\} }$, this implies that the equation \eqref{Stefan-nu}  is reduced to 
 $(St)$, proving that the Eulerian variable for $\tau^*$ gives the solution to $(St)$.

  Now let us show that $\eta$ is maximal, in the sense of Definition~\ref{def:maximal}.
 Suppose that there is a weak solution $\tilde \eta$ with the same initial data $(\tilde \eta_0, E)=(\mu, U)$ such that 
   $\nu^*=\chi_{\Sigma(\eta)}\le_{SH} \chi_{\Sigma(\tilde \eta)}.$  Define $\tilde \nu: = \chi_{\Sigma(\tilde \eta)}$.  By definition of subharmonic ordering  there is a stopping time, say, $\sigma$, with $W_0 \sim \nu^*$ and $W_\sigma \sim \tilde \nu$, and $\mathbb{E}[\sigma] <\infty$. 
 So, $\tilde \nu$ is admissible for $P(\mu, f=1, u)$, and from the subharmonic order and \eqref{eqn:maximum-good} we have 
  $\int u d\tilde \nu \le \int u d\nu^*$.  Now the uniqueness of optimal solutions for $P(\mu, f=1, u)$ (Corollary~\ref{cor:uniqueness-primal}) implies 
  $$
  \nu^*=\tilde \nu.
  $$ 
   Our goal is to show that $\tilde \eta =\eta$. Here we will use the characterization of  the weak solution $\tilde{\eta}$  as an  Eulerian variable associated with the optimal Skorokhod problem, proven in \cite{Kim-Kim21}. To this end let us first introduce the corresponding $\rho$-variable for $\tilde{\eta}$.
We  define 
\begin{align*}
 \tilde s(x):= \sup \{t: \tilde \eta(x,t)>0\} \hbox{ and } 
\end{align*}
 $$
 \int_0^\infty\int_{\R^d} \tilde \rho(x) \varphi(x,t) \, dx dt =  \int \tilde{\nu}(x)\varphi(x,\tilde s(x)) dx \hbox{  for any  } \varphi\in C^{\infty}_c(\R^d\times [0,\infty)).
 $$
 Due to \cite[Theorem 5.6]{Kim-Kim21}, then 
  $(\tilde \eta,\tilde \rho)$ is the Eulerian variables between $\tilde \eta_0=\mu$ and $\tilde{\nu}$, generated by the optimal stopping time, say $\tilde \tau$,  of the optimal Skorokhod problem  in the sense of \cite{PDE19}  for costs of type (I). 
 Since $\tilde \nu = \nu^*$ we have $\tilde \tau =\tau^*$  almost surely, from the uniqueness of such an optimal stopping time \cite{PDE19}. 
This implies that  their corresponding Eulerian variables coincide, in particular, $\eta =\tilde \eta$, and we conclude.
\end{proof}

\begin{remark}\label{rmk:condition}
Note that in the proof of Theorem~\ref{th:Stefan-weak}, The assumption $|\{\eta_0=1\} |=0$ is used to ensure that the associated optimal stopping time is strictly positive. This is necessary to ensure that the active particle density $\eta$ that is obtained from $P(\mu, f=1, u)$, which is our solution of $(St)$, assumes its initial data as $\eta_0$.
\end{remark}

 \begin{remark}\label{continuity}
Note that our definition requires $\eta$ to solve the heat equation in the interior of  its positive set in $\R^d \times [0,\infty)$. Thus, formally if the set $\{\eta>0\}$ discontinuously shrinks at time $t=t_0$,  then $\eta(\cdot,t_0^-)\leq \eta(\cdot,0) <1$. On the other hand, $\{\eta(\cdot,t_0^-)>0\} \setminus\{\eta(\cdot,t_0^+)>0\}$ is a subset of $\{s(x) = t_0\}$, where  $s(x) := \inf \{ t :  (x,t) \in \{\eta>0\}^c \}.$ Thus there $\eta(\cdot,t_0^-)$ equals $\nu=1$.  
Hence we conclude that the set $\{\eta(\cdot,t)>0\}$ cannot decrease discontinuously. Of course this heuristic argument does not carefully consider the potential irregularity of $\partial\{\eta(\cdot,t)>0\}$.  The precise statement of the ``no-jump" property remains to be established.
\end{remark}

\subsection{Irregular evolution is generic for $(St)$}\label{observation00}
The  free boundary for a solution of  $(St)$ can be highly irregular in its dynamic evolution, if the initial free boundary $\partial\{\eta_0>0\}$ is less than $C^{\infty}$. More precisely, any presence of a non-smooth point on the initial free boundary may cause either global-time waiting time phenomena or instant nucleation of additional boundary points in its evolution. We explain it below:

\medskip

In \cite[Theorem 4.15]{Kim-Kim21}  it is shown that the optimal stopping time $\tau$ between  given initial and final measures $\mu$ and $\nu$, is determined by its  potential flow $U_t$ (with $\Delta U_t  = \mu_t$ and $\Delta U_\infty =\nu$, for $W_{t\wedge \tau} \sim \mu_t$). This implies that  if $\eta$ is the Eulerian variable for the initial distribution $\mu$ and target measure $\nu$, then the active region $\{x:\eta(x,t)>0 \hbox{ for some } t>0\}$ is given by $\{w>0\}$, where $w$ solves 
$$
\Delta w = \nu - \mu.
$$
 If the associated optimal stopping time $\tau$ is positive, then both $\nu$ and $\mu$ are supported in $\{w>0\}$, and thus one can write above equation as 
$$
 \Delta w = (\nu-\mu)\chi_{\{w>0\}}.
$$
In our problem, $U = \{\eta_0>0\}$ equals the active region of $\eta$,  which is $\{w>0\}$. Moreover we show  that the optimal target measure $\nu=\nu^*$ is a characteristic function $\chi_{\Sigma}$, where $\Sigma$ is the trace of the free boundary in space, namely $\Sigma= \{x: \eta(x,t)=0\hbox{ for some } 0<t<\infty\}\cap U$. Thus if we suppose
\begin{equation}\label{assumption0}
 \partial \Sigma \cap B_r(x_0) = \partial U \hbox{ for some }r>0 \hbox{ and } x_0\in\partial U,
 \end{equation}
 then we conclude that $w$ solves 
$$
 \Delta w = (1-\eta_0)\chi_{\{w>0\}} \hbox{ in } B_r(x_0),
$$
where now $\eta_0$ denotes its density. 

Suppose that $\eta_0$ is in $C^k(\bar{U})$ for some $k\in \mathbb{N}$, and strictly less than $1$. Then this problem is now the classical obstacle problem with smooth data. Hence the local regularity theory available for this problem conditions the geometry of  $\partial\{w>0\}=\partial U$ near $x_0$. Since $U$ is a domain with locally Lipschitz boundary, it follows from classical literature on the obstacle problem (for instance \cite{KN77}) that  $\partial\{w>0\}$ is $C^{k+1}$ in $B_r(x_0)$. So we conclude that, for instance, \eqref{assumption0} cannot be true for any boundary points of $U$ that does not satisfy this regularity. Then for the simple case where $\eta_0$ is a constant in $U$,  $U$ must have $C^{\infty}$ boundary to satisfy \eqref{assumption0}, otherwise we expect instant cracks appear on the free boundary, when the initial flux $|D\eta_0|$ at $x_0$ is nonzero.

 To give more explicit examples, let us assume that $U$ has a corner at $x_0$ and  $\eta_0$ is harmonic in $U$ with zero boundary conditions. When the corner at $x_0$ points outward (sharp corner), then the initial flux of $\eta$ is zero at $x_0$ and we expect $\partial\{\eta(\cdot,t)>0\}$ to still contain $x_0$ for all times. When the corner at $x_0$ points toward inward (obtuse corner), then the initial flux of $\eta$ is infinite at $x_0$ and we expect an instant slit to form on the boundary of $\{\eta(\cdot,t)>0\}$, propagating from $x_0$. In two space dimensions, this slit blocks Brownian particles, and they stop before hitting it (almost surely). Therefore, the initial domain may not be the actual region for the Brownian particle dynamics as there is a measure zero set of difference.  Similar phenomena may occur near non-smooth boundary point, in space dimension $d\ge 3$, instantly generating  a $d-1$ dimensional slit inside the original domain, that blocks Brownian motion to pass. These possible phenomena are accounted in our Theorem~\ref{thm:A}, with the expression ``possibly minus a measure zero set."

\section{Examples of initial data for existence and non-existence results.}\label{sec:obstacle-to-Stefan}

This section discusses examples of initial data $\eta_0$ to which our theory applies. More precisely we provide concrete examples of $\eta_0$ which either fails $(C)$ or satisfies $(C_0)$ but is larger than $1$ in some parts. Due to our results in the previous section, these examples yield in turn the existence and non-existence results for $(St)$. The examples we provide are given in terms of elliptic obstacle problems associated to $\eta_0$ (see also Remark~\ref{observation00}). 

\medskip

First we present a concrete condition on initial data $\eta_0$, in terms of an associated elliptic obstacle problem, so that it fails condition $(C)$ in the introduction. Theorem~\ref{thm:necessary} then yields a nonexistence result.

\begin{corollary}[Nonexistence]\label{cor:nonexistence}
Given $E\subset \R^d$ and $\eta_0\in C(\R^d)$ supported in $E$, 
consider the solution $w$ of the obstacle problem:
\begin{align*}
 \Delta w = (1- \eta_0)\chi_{\{w >0\}},  
 \quad w \ge 0. 
\end{align*}
Suppose that 
\begin{equation}\label{condition:0}
\emptyset \neq \{w >0\} \not \subseteq E.
\end{equation}
 Then $\eta_0$ fails $(C)$, and there is no weak solution to $(St)$  with the initial data $(\eta_0, E)$. 

\end{corollary}

\begin{proof}
We show that the inclusion $\{w >0\}  \subseteq E$ is necessary for the existence of weak solutions. 
Suppose $(St)$ has a weak solution $\eta$ in this setting. Theorem~\ref{thm:necessary} implies that, condition $(C)$ holds for $\eta_0$, and thus there is a target measure $\nu= \chi_A$ with a stopping time $\tau$ with $W_0 \sim \eta_0$ and $W_\tau \sim \nu$. 
Now use \cite[Corollary 8.5]{Kim-Kim21} 
 to see that $$\bar \nu := \chi_{\{w>0\}} + \eta_0\chi_{\{w =0\}}$$ is the unique solution for the optimization problem considered in \cite[Theorem 7.4]{Kim-Kim21} with the choice $\mu= \eta_0$ and the upper bound constraint $f=1$ there. In particular $\bar \nu \le 1$.
  Let $\bar \tau$ be the optimal stopping time (in the sense of \cite[Theorems 7.4]{Kim-Kim21}) between $\eta_0$ and $\bar \nu$: from \cite[Theorems 8.3]{Kim-Kim21} we see that the active region for the associated active particle density $\bar{\eta}$, namely the set $\{\bar{\eta}>0\}$, satisfies $\lim_{\{t\to 0^+\}}\{\bar{\eta}(\cdot,t)>0\} = \{w>0\}$.
   Since $\eta_0\chi_{\{ w>0\}} \le \eta_0$, the monotonicity result  \cite[Theorem 7.1(1)]{Kim-Kim21} gives that 
$$\bar \tau \le \tau \hbox{ almost surely.}$$ Thus it follows that $\{\bar{\eta}>0\}\subset \{\eta>0\}$.
By sending $t\to 0^+$, it follows that $$\{w >0\} \subseteq E.$$ 
This completes the proof. 
\end{proof}

\begin{remark} Note that if $\eta_0 \leq1$, then $w\equiv 0$ and thus \eqref{condition:0} fails. On the other hand it is easy to see from the elliptic equation that $w$ should be differentiable and in particular has its slope zero when it hits zero value. Hence if $\eta_0$ is large on most of $E$, $w$ is very concave in most of $E$. For instance in the radial case, one can imagine that $w$ has large downward slope, growing linearly away from the center point of $E$. From this and the fact that $w$ is $C^2$, it is easy to check that $\{w>0\}$ needs to be larger than $E$.
\end{remark}

We next provide conditions under which $\eta_0$ satisfies condition $(C_0)$ with some $\nu$,  based on the results from \cite{Kim-Kim21}. Existence results follow due to Theorem~\ref{thm:A}.
\begin{corollary}[Existence]\label{cor:existence}
Given $\eta_0 \in L^\infty(\R^d)$, assume that the set $\{ \eta_0 >0\}$ is a bounded open set with locally Lipschitz boundary with the property $| \{ \eta_0 =1\}| =0$.
 Suppose  further that there exists  $0<\delta <1$ such that for the solution $w$ of the 
  obstacle problem
\begin{align*}
 \Delta w = ( \delta -   \eta_0)\, \chi_{\{w  >0\}} , \quad w \ge 0, 
\end{align*}
satisfies
\begin{align*}
 \{ w>0\} \subseteq \{ \eta_0 >0\}. 
\end{align*}
Then, there is a weak solution to $(St)$ with the initial data $ \eta_0$ and initial domain $\{\eta_0>0\}$ . \end{corollary}
\begin{proof}
The proof is in the same spirit of the one for Corollary~\ref{cor:nonexistence}.
Use \cite[Corollary 8.5]{Kim-Kim21} 
 to see that $$ \eta_0 \le_{SH} \bar \nu := \delta \chi_{\{w>0\}} +  \eta_0\chi_{\{w =0\}}$$ where $\bar \nu$ is the unique optimal solution in the sense of \cite[Theorem 7.4]{Kim-Kim21} with the choice $\mu= \eta_0$ and $f=\delta$ there; so $\bar \nu \le \delta$.
 Moreover, from the condition $\{w >0\} \subseteq \{  \eta_0>0\}$ it holds that $\{\bar \nu >0\} =\{  \eta_0 >0\}$. 
 We thus can apply Theorem~\ref{thm:A} as $\eta_0$ and $\bar \nu$ satisfy the condition $(\mathcal{C}_0)$. This completes the proof.
 \end{proof}
\begin{remark}\label{rmk:many-examples}
 Notice that the condition of Corollary~\ref{cor:existence} can easily be satisfied by many $\eta_0$ that takes values larger than $1$, for instance  when $\eta_0$ is large in the center of its support, far away from its boundary. \end{remark}

\end{document}